\documentclass[a4paper]{amsart}

\usepackage{amsfonts,amsmath}
\usepackage{amsthm,amsxtra,amssymb}
\usepackage{mathrsfs}
\usepackage[pagebackref,colorlinks,linkcolor=blue,citecolor=blue,urlcolor=blue]{hyperref}
\usepackage[all]{xy}
\usepackage{color}
\usepackage{stmaryrd}
\usepackage[vcentermath]{youngtab}
\usepackage{aliascnt}
\usepackage{setspace}
\usepackage{graphicx}

\DeclareMathOperator{\im}{im}

\DeclareMathOperator{\Hom}{Hom}

\providecommand{\id}{\ensuremath\mathrm{id}}

\DeclareMathOperator{\Aut}{Aut}

\DeclareMathOperator{\Inn}{Inn}
\DeclareMathOperator{\Out}{Out}

\DeclareMathOperator{\GL}{GL}
\DeclareMathOperator{\SL}{SL}
\DeclareMathOperator{\IA}{IA}
\DeclareMathOperator{\Mod}{Mod}
\DeclareMathOperator{\Homeo}{Homeo}
\providecommand{\I}{\ensuremath\mathcal{I}}
\DeclareMathOperator{\Sp}{Sp}

\DeclareMathOperator{\modGL}{modGL}
\DeclareMathOperator{\modSp}{modSp}

\DeclareMathOperator{\Ind}{Ind}
\DeclareMathOperator{\Res}{Res}

\DeclareMathOperator{\gr}{gr}
\DeclareMathOperator{\tr}{tr}
\providecommand{\FI}{\ensuremath\mathsf{FI}}
\providecommand{\VI}{\ensuremath\mathsf{VI}}
\providecommand{\VIC}{\ensuremath\mathsf{VIC}}
\providecommand{\SI}{\ensuremath\mathsf{SI}}
\providecommand{\C}{\ensuremath\mathcal{C}}

\providecommand{\xmod}[1]{\ensuremath{#1\mathsf{-mod}}}

\providecommand{\hooklongrightarrow}{\lhook\joinrel\longrightarrow}
\providecommand{\twoheadlongrightarrow}{\relbar\joinrel\twoheadrightarrow}
\providecommand{\inject}{\hooklongrightarrow}
\providecommand{\surject}{\twoheadlongrightarrow}

\providecommand{\quot}[2]{{\raisebox{.2em}{$#1\!$}\left/\raisebox{-.2em}{$#2$}\right.}}
\newcommand{\doublequot}[3]{{\left.\raisebox{-.2em}{$#1\!$}\right\backslash\raisebox{.2em}{$\!#2\!$}\left/\raisebox{-.2em}{$\!#3$}\right.}}
\newcommand{\tens}[1][]{\mathbin{\mathop{\otimes}\displaylimits_{#1}}}
\providecommand{\NN}{\ensuremath\mathbb N}
\providecommand{\ZZ}{\ensuremath\mathbb Z}
\providecommand{\QQ}{\ensuremath\mathbb Q}
\providecommand{\RR}{\ensuremath\mathbb R}

\definecolor{grey}{gray}{.5}

\numberwithin{thmcounter}{section}
\newaliascnt{thmauto}{thmcounter}

\newaliascnt{Defauto}{thmcounter}

\newaliascnt{exauto}{thmcounter}

\newaliascnt{lemauto}{thmcounter}

\newaliascnt{propauto}{thmcounter}

\newaliascnt{corauto}{thmcounter}

\newaliascnt{remauto}{thmcounter}

\theoremstyle{plain}
\newtheorem{thm}[thmauto]{Theorem}
\newtheorem{Def}[Defauto]{Definition}
\newtheorem{ex}[exauto]{Example}
\newtheorem{lem}[lemauto]{Lemma}
\newtheorem{prop}[propauto]{Proposition}
\newtheorem{cor}[corauto]{Corollary}
\newtheorem{rem}[remauto]{Remark}
\newtheorem{thmA}{Theorem}

\newtheorem*{rem*}{Remark}
\newtheorem*{thm*}{Theorem}

\numberwithin{equation}{section}

\let\originalleft\left
\let\originalright\right
\renewcommand{\left}{\mathopen{}\mathclose\bgroup\originalleft}
\renewcommand{\right}{\aftergroup\egroup\originalright}

\title{Representation stability for filtrations of Torelli groups}

\author{Peter Patzt}
\address{Institut f\"ur Mathematik, Freie Universit\"at Berlin, Germany}
\email{peter.patzt@fu-berlin.de}
\date{April 2017}
\subjclass[2010]{20G05 (Primary), 20E36, 20F40, 58D05 (Secondary)}

\begin{document}

\maketitle 

\begin{abstract}
We 
show, finitely generated rational $\VIC_\QQ$--modules and $\SI_\QQ$--modules are uniformly representation stable and all their submodules are finitely generated. We use this to prove two conjectures of Church and Farb, which state that the quotients of the lower central series of the Torelli subgroups of $\Aut(F_n)$ and $\Mod(\Sigma_{g,1})$ are uniformly representation stable as sequences of representations of the general linear groups and the symplectic groups, respectively. Furthermore we prove an analogous statement for their Johnson filtrations.
\end{abstract}

%


\section{Introduction}

%

Church and Farb \cite{CF} define the notion of representation stability for sequences of representations of the symmetric groups $\mathfrak S_n$, the hyperoctahedral groups $\ZZ\rtimes \mathfrak S_n$, the general linear groups $\GL_n\QQ$, the special linear groups $\SL_n\QQ$ and the symplectic groups $\Sp_{2n}\QQ$. Especially representation stability for the symmetric groups has been the focus of a lot of research lately. It has been intimately connected to functors from the category of finite sets and injections $\FI$ to vector spaces by Church, Ellenberg and Farb \cite{CEF}. Wilson \cite{Wi} developed a similar connection for the hyperoctahedral groups. In both cases an amplitude of sequences were proved to be representation stable.

\subsection*{Representation stability over the symmetric groups.}
The representation theory of the symmetric group $\mathfrak S_n$ over the rationals $\QQ$ is known to be semisimple. The irreducible representations are indexed by partitions $\lambda=(\lambda_1\ge\lambda_2\ge \dots)$ of $n= \lambda_1+ \lambda_2 + \dots$, which we denote by
\[ \mathfrak S_n(\lambda).\]
Let 
\[ V_0 \stackrel{\phi_0}{\longrightarrow} V_1 \stackrel{\phi_1}{\longrightarrow} V_2 \stackrel{\phi_2}{\longrightarrow} \dots\]
 be a sequence of vector spaces over $\QQ$ together with a linear $\mathfrak S_n$--action on $V_n$ such that $\phi_n$ is $\mathfrak S_n$--equivariant. Such a sequence is called \emph{consistent} by Church--Farb \cite{CF} and can easily be generalized to the other groups mentioned in the first paragraph. 
They \cite[Def 2.3]{CF} call a consistent sequence of representations of the symmetric groups \emph{representation stable} if the following conditions are satisfied:
\begin{description}
\item[Injectivity] The map $\phi_n\colon V_n \to V_{n+1}$ is injective for all large enough $n\in\mathbb N$.
\item[Surjectivity] The induced map $\Ind_{\mathfrak S_n}^{\mathfrak S_{n+1}} \phi_n\colon \Ind_{\mathfrak S_n}^{\mathfrak S_{n+1}}V_n \to V_{n+1}$ is surjective for all large enough $n\in \mathbb N$.
\item[Multiplicity stability] We can write
\[ V_n \cong \bigoplus_{\lambda} \mathfrak S_n(\lambda)^{\oplus c_{\tilde\lambda,n}}\]
where $\tilde\lambda = (\lambda_2\ge\lambda_3\ge \dots)$ for $\lambda = (\lambda_1\ge\lambda_2\ge \lambda_3\ge \dots)$ and $c_{\tilde\lambda,n}$ is independent of $n$ for all large enough $n\in \mathbb N$.
\end{description}
A consistent sequence is called \emph{uniformly} representation stable if the multiplicities $c_{\tilde\lambda,n}$ stabilize uniformly.

Functors from a category $\C$ to the category $\xmod\QQ$ of vector spaces over $\QQ$ are called \emph{$\C$--modules}. Every $\FI$--module $V\colon \FI\to \xmod\QQ$ gives rise to a consistent sequence, by taking
\[ V_n = V( \{ 1, \dots, n\})\]
and
\[ \phi_n = V( \{ 1, \dots, n\} \to \{ 1, \dots, n+1\}).\]
The connection to representation stability was provided by Church--Ellenberg--Farb in the following theorem.

\begin{thm*}[Church--Ellenberg--Farb {\cite[Thm 1.13]{CEF}}]
An $\FI$--module $V$ is finitely generated if and only if its consistent sequence is uniformly representation stable and $V_n$ is finite dimensional for all $n\in\NN$.
\end{thm*}

This theorem depends on the following noetherian property of $\FI$--modules. 

\begin{thm*}[Church--Ellenberg--Farb {\cite[Thm 1.3]{CEF}}]
Every submodule of a finitely generated $\FI$--module is finitely generated.
\end{thm*}

Analogous theorems for the hyperoctahedral groups were proved by Wilson \cite[Thm 4.21 + Thm 4.22]{Wi}.


\subsection*{Representation stability over the general linear groups and symplectic groups.}
The rational representation theory for both $\GL_n\QQ$ and $\Sp_{2n}\QQ$ is semisimple and the irreducibles are indexed by pairs of partitions $(\lambda^+, \lambda^-)$ such that the lengths $\ell(\lambda^+) + \ell(\lambda^-)\le n$ and by partitions $\lambda$ whose length $\ell(\lambda) \le n$, respectively. We respectively denote these irreducibles by
\[ \GL_n(\lambda^+,\lambda^-)\quad\text{and}\quad \Sp_{2n}(\lambda).\]
For a consistent sequence of rational representations of the general linear groups or the symplectic groups Church--Farb \cite[Def 2.3]{CF} define (uniform) representation stability analogously to the symmetric groups---only the analogue of the multiplicity stability condition is easier to state:
\begin{description}
\item[Multiplicity stability for general linear groups] We can write
\[ V_n \cong \bigoplus_{\lambda^+,\lambda^-} \GL_n(\lambda^+,\lambda^-)^{\oplus c_{(\lambda^+,\lambda^-),n}}\]
and $c_{(\lambda^+,\lambda^-),n}$ is independent of $n$ for all large enough $n\in \mathbb N$.
\item[Multiplicity stability for symplectic groups] We can write
\[ V_n \cong \bigoplus_{\lambda} \Sp_{2n}(\lambda)^{\oplus c_{\lambda,n}}\]
and $c_{\lambda,n}$ is independent of $n$ for all large enough $n\in \mathbb N$.
\end{description}

The question of the analogue of $\FI$ for the general linear groups could be naively answered with $\VI$---the category of finite dimensional vector spaces and injective homomorphisms. But this is not correct, it turns out that the correct analogue is $\VIC$---the category of finite dimensional  vector spaces and injective homomorphisms together with a choice of a complement of the image (see \autoref{Def:VIC}). For the symplectic groups we use $\SI$---the category of finite dimensional  symplectic vector spaces and isometries. This works well, because isometries are always injective and come with a canonical complement. A version of $\VIC$ and $\SI$ for finite rings has already been used by Putman--Sam \cite{PS}. 

Every $\VIC$--module $V\colon \VIC\to \xmod\QQ$ gives rise to a consistent sequence, by taking
\[ V_n = V(\QQ^n)\]
and
\[ \phi_n = V( \QQ^n \to \QQ^{n+1}).\]
Similarly for every $\SI$--module $V\colon \SI \to \xmod\QQ$ the sequence given by
\[ V_n = V(\QQ^{2n})\]
and
\[ \phi_n = V( \QQ^{2n} \to \QQ^{2n+2})\]
is consistent. We call $V$ \emph{rational} if $V_n$ is a rational representation for every $n\in\NN$.

Our main technical results are the following theorems.

\begin{thmA}\label{thmA:VICrepstab}
A rational $\VIC$--module $V$ is finitely generated if and only if its consistent sequence is uniformly representation stable and $V_n$ is finite dimensional for all $n\in\NN$.
\end{thmA}

\begin{thmA}\label{thmA:SIrepstab}
A rational $\SI$--module $V$ is finitely generated if and only if its consistent sequence is uniformly representation stable and $V_n$ is finite dimensional for all $n\in\NN$.
\end{thmA}

We also prove the following noetherian condition.

\begin{thmA}\label{thmA:VICnoeth}
Every submodule of a finitely generated rational $\VIC$--module is finitely generated.
\end{thmA}

\begin{thmA}\label{thmA:SInoeth}
Every submodule of a finitely generated rational $\SI$--module is finitely generated.
\end{thmA}

\begin{rem*}
\begin{enumerate}
\item We may substitute any field of characteristic zero for $\QQ$ and the theorems remain true. 
\item Putman--Sam \cite{PS} proved analogues of \hyperref[thmA:VICnoeth]{Theorems \ref{thmA:VICnoeth}} and \ref{thmA:SInoeth} for finite rings.
\item Gan--Watterlond \cite{GW} proved an analogue of \autoref{thmA:VICrepstab} for finite fields. 
\end{enumerate}
\end{rem*}

\subsection*{Torelli groups.}
Let $F_n$ denote the free group on $n$ generators, then its abelianization is $\ZZ^n$. The quotient map induces an epimorphism on their automorphism groups.
The {Torelli subgroup $\IA_n$} is defined as the kernel, so we get the following short exact sequence.
\[ 1 \to \IA_n \to \Aut(F_n) \to \Aut(\ZZ^n) \cong \GL_n \ZZ \to 1\]

Let $\Sigma_{g,1}$ denote the compact, oriented genus $g$ surface with one boundary component. The mapping class group $\Mod(\Sigma_{g,1})$ is the discrete group $\pi_0 \Homeo^+(\Sigma_{g,1},\partial\Sigma_{g,1})$ of isotopy classes of orientation-preserving homeomorphisms of $\Sigma_{g,1}$ that fix the boundary pointwise. The action of $\Mod(\Sigma_{g,1})$ on $H_1(\Sigma_{g,1};\ZZ)\cong \ZZ^{2g}$ is symplectic, and the {Torelli subgroup} $\I_{g,1}$ is defined to be the kernel of this action. In fact, there is a short exact sequence
\[ 1 \to \I_{g,1} \to \Mod(\Sigma_{g,1}) \to \Sp(H_1(\Sigma_{g,1};\ZZ)) \cong \Sp_{2g}(\ZZ) \to 1.\]

Very little is known about the homology of both Torelli subgroups, except in homological degree $1$. The rational homology is conjectured to be uniformly representation stable in \cite[Conj 6.1, Conj 6.3]{CF}. This problem seems to be too hard to tackle right now, as it is not even known whether the rational homology groups are representations of $\GL_n\QQ$ and $\Sp_{2g}\QQ$. 

Another subject of study deals with central series of the Torelli groups, which include the lower central series $ \gamma \IA_n = \{\gamma_i \IA_n\}_{i\in\NN} $ and $\gamma \I_{g,1}$ and  the Johnson filtration $\alpha \IA_n$ and $\alpha\I_{g,1}$ (see \autoref{sec:N-series} and the beginning of  \autoref{section:I}). The information of these central series are compiled nicely in their graded rational Lie algebra $\gr(\gamma\IA_n)$, $\gr(\gamma\I_{g,1})$, $\gr(\alpha\IA_n)$ and $\gr(\alpha\I_{g,1})$ (see \autoref{Def:gr}). All of these filtrations were considered before, for example by Andreadakis \cite{An}, Hain \cite{Hai}, Habegger--Sorger \cite{HS}, Satoh \cite{Sat,SatSurvey} and are known to be separating, ie
\[ \bigcap_{i\ge 1} \gamma_i\IA_n = \bigcap_{i\ge 1} \gamma_i\I_{g,1} = \bigcap_{i\ge 1} \alpha_i\IA_n = \bigcap_{i\ge 1} \alpha_i\I_{g,1} = 1.\] 
Church and Farb  \cite[Conj 6.2 and the paragraph below Conj 6.3]{CF} conjectured  that each degree of the Lie algebras corresponding to the lower central series is uniformly representation stable. The following theorems address exactly the conjectures as stated by Church and Farb.

\begin{thmA}\label{thmA}
For every fixed $i\ge1$ and $n\in \NN$, the natural $\GL_n\ZZ$--representation on the $i$th quotient of the lower central series $\gr_i(\gamma\IA_n)$ extends to a rational $\GL_n \QQ$--representation.
\end{thmA}

\begin{thmA}\label{thmB}
For every fixed $i\ge 1$ the sequence of the $i$th quotients of the lower central series $\{\gr_i(\gamma\IA_{n})\}_{n\in\NN}$ of $\GL_n \QQ$--representations is uniformly representation stable.
\end{thmA}

\begin{thmA}\label{thmC}
For every fixed $i\ge 1$ the sequence of the $i$th quotients of the lower central series $\{\gr_i(\gamma\I_{g,1})\}_{g\in\NN}$ of $\Sp_{2g} \QQ$--representations is uniformly representation stable.
\end{thmA}

It is noteworthy that if a $\GL_n\ZZ$--representation can be extended to a rational $\GL_n\QQ$--representation, this extension is not unique. In \autoref{sec:res} it is explained how there are infinitely many different possible extensions. However, a sequence of extensions that satisfies \autoref{thmB} is uniquely determined for all large enough $n\in\NN$. To prove \autoref{thmB}, we will find the correct way to extend these representations for large enough $n\in\NN$. For representations of the symplectic groups this problem does not arise.

We are also able to prove similar results for the Lie algebras corresponding to the Johnson filtrations.

\begin{thmA}\label{thmD}
For every fixed $i\ge 1$ the sequence of the $i$th quotients of the Johnson filtration $\{\gr_i(\alpha\IA_{n})\}_{n\in\NN}$ of $\GL_n \QQ$--representations is uniformly representation stable.
\end{thmA}

\begin{thmA}\label{thmE}
For every fixed $i\ge 1$ the sequence of the $i$th quotients of the Johnson filtration $\{\gr_i(\alpha\I_{g,1})\}_{g\in\NN}$ of $\Sp_{2g} \QQ$--representations is uniformly representation stable.
\end{thmA}

 We also prove analogues of \autoref{thmA} for the filtrations $\gamma\I_{g,1}$, $\alpha\IA_n$ and $\alpha\I_{g,1}$, although they can already be found in the literature (eg in \cite[Thm 1.1]{HS} and \cite{SatSurvey}).

This work is the author's PhD thesis.

\subsection*{Acknowledgements.}
First and foremost the author wishes to thank his advisor Holger Reich for introducing him to the interesting and emerging research on representation stability. During his PhD the author was supported by the Berlin Mathematical School, the SFB Raum--Zeit--Materie and the Dahlem Research School. The author also wants to thank Kevin Casto, Tom Church, Daniela Egas Santander, Benson Farb, Daniel L\"utgehetmann, Jeremy Miller, Holger Reich, Steven Sam, David Speyer and Elmar Vogt for helpful conversations. Special thanks to Steven Sam for his extensive help with the modification rules, and to Kevin Casto for pointing out the conjectures to the author.


\section{Rational representation theory of the general linear groups and the symplectic groups}\label{section:repthy}

Let us start by shortly recalling the rational (or algebraic) representation theory of the algebraic groups $\GL_n\QQ$ and $\Sp_{2n}\QQ$. More elaboration can be found in the books of Fulton--Harris \cite{FH}, Green \cite{Green}, Goodman--Wallach \cite{GW09}, Jantzen \cite{Jantzen}, Weyl \cite{Weyl}, and the paper of Koike \cite{Ko}.

\subsection{Algebraic groups, polynomial and rational representations}\label{sec:alggrp}
In general an \emph{algebraic group} over a field $k$ is a variety that has a compatible group structure. That means multiplication and inverses are regular maps of varieties. Two simple examples are the additive group $(k,+)$ considered as the affine variety $\mathbb A^1$ and the multiplicative group $(k^\times,\cdot)$ considered as the subvariety of $\mathbb A^2$ given by the polynomial $xy=1$.

A more complicated example is the \emph{general linear group} $\GL_n(k)$. To define it, we consider a subvariety of $\mathbb A^{n^2+1}$. Let its coordinates be denoted by $\{x_{ij}\}_{i,j= 1, \dots, n}$ and $t$. Then the determinant $\det(x_{ij})$ of the matrix given by $\{x_{ij}\}_{i,j = 1, \dots, n}$ is a polynomial. Let $\GL_n(k)$ be the subvariety of $\mathbb A^{n^2+1}$ given as the zero set of the polynomial $\det (x_{ij})\cdot t -1$. The multiplication given by matrix multiplication and the inverse given by Cramer's rule is polynomial. Thus $\GL_n(k)$ is an algebraic group. 

The \emph{symplectic group} $\Sp_{2n}(k)$ is the subgroup of $\GL_{2n}(k)$ given by those matrices $(x_{ij})$ whose inverse is
\[ (x_{ij})^{-1} =  \Omega_n \cdot (x_{ij}^T) \cdot \Omega_n^T\]
where $\Omega_n$ is the Gram matrix of the standard symplectic form
\[ \Omega_n = \begin{pmatrix} 0&1 \\ -1&0 \\ &&\ddots\\ &&&0&1 \\ &&&-1&0 \end{pmatrix}.\]
Therefore the symplectic group is the zero set of the polynomials
\[ (x_{ij})\cdot  \Omega_n \cdot (x_{ij}^T) \cdot \Omega_n^T -1.\]
Thus $\Sp_{2n}(k)$ is an algebraic group.

A (finite dimensional) \emph{polynomial/rational representation} of an algebraic group $G$ over $k$ is a map
\[ G\longrightarrow \GL_n(k)\]
for some $n\in\NN$ that is a group homomorphism and a polynomial/rational map of varieties. For infinite dimensional representations, one may construct the algebraic group $\GL(V)$ for infinite dimensional vector spaces $V$ over $k$. Note that every polynomial representation is by definition also rational.

Both for $\GL_n(k)$ and $\Sp_{2n}(k)$ there is a \emph{standard representation} given by
\[ \GL_n(k) \stackrel{\id}{\longrightarrow} \GL_n(k)\]
and
\[ \Sp_{2n}(k) \inject \GL_{2n}(k)\]
respectively. Both are polynomial representations.

\subsection{The representation theory of $\GL_n\QQ$}
It turns out that both the polynomial representation theory and the rational representation theory of $\GL_n\QQ$ are semisimple and all irreducible representations are finite dimensional. 

The standard representation $V_n=\QQ ^n$ is irreducible.  All other irreducible {polynomial representations} can be constructed as subquotients of the $r$-fold tensor product $V_n^{\otimes r}$ for some $r\in\NN$ on which $\GL_n\QQ$ acts diagonally. We get a right action of the symmetric group $\mathfrak S_r$ on $r$ letters on $V_n^{\otimes r}$, which makes it a $\QQ \GL_n\QQ$--$\QQ \mathfrak S_r$--bimodule. Let $\lambda$ be a partition of $r$, then $r$ is called the \emph{size} of $\lambda$ and is denoted by $|\lambda|$. Let $\mathfrak S_r(\lambda)$ be its associated irreducible Specht module of $\mathbb Q\mathfrak S_r$, then
\[ \GL_n(\lambda) := V_n^{\otimes r} \tens[\QQ \mathfrak S_r] \mathfrak S_r(\lambda)\]
is an irreducible $\GL_n\QQ$--representation if $\lambda$ has at most $n$ rows and zero otherwise. 
We call this number the \emph{length} of $\lambda$ and denote it by $\ell(\lambda)$. 
In fact, all irreducible polynomial $\GL_n\QQ$--representations are isomorphic to $\GL_n(\lambda)$ for some partition $\lambda$ with at most $n$ rows and they are up to isomorphism uniquely determined by it. 


To get {rational representations} of $\GL_n\QQ$, we need to introduce the dual representation $V^*_n = \Hom_\QQ (\QQ ^n,\QQ )$ of $V_n$, which is defined by $g\cdot f(v) = f(g^{-1}\cdot v)$. Define furthermore $V_n^{\{r,s\}}$ to be the intersection of the kernels of all contraction maps
\begin{align*} 
V_n^{\otimes r}\otimes {V^*_n}^{\otimes s} &\longrightarrow V_n^{\otimes r-1}\otimes {V^*_n}^{\otimes s-1}\\ 
v_1\otimes \dots \otimes v_r \otimes f_1\otimes \dots \otimes f_s & \longmapsto f_j(v_i) \cdot v_1 \otimes \dots \otimes \hat v_i \otimes \dots \otimes v_r \otimes f_1 \otimes \dots \otimes \hat f_j \otimes \dots \otimes f_s.\
\end{align*}
Let $\lambda^+$ be a partition of $r$ and $\lambda^-$ a partition of $s$. We call $r+s$ the size of the pair $(\lambda^+,\lambda^-)$. Then
\[ \GL_n(\lambda^+,\lambda^-) := V_n^{\{r,s\}} \tens[\QQ \mathfrak S_r \otimes \QQ \mathfrak S_s] \big( \mathfrak S_r(\lambda^+) \otimes \mathfrak S_s(\lambda^-)\big)\]
is a rational $\GL_n\QQ$--representation. It is irreducible if the length of the pair $\ell(\lambda^+)+\ell(\lambda^-)\le n$ and zero otherwise. All irreducible rational $\GL_n\QQ$--rep\-res\-en\-ta\-tions are isomorphic to $\GL_n(\lambda^+,\lambda^-)$ for some partitions $\lambda^+,\lambda^-$ which  together  have at most $n$ rows and they are up to isomorphism uniquely determined by it. 

In terms of weights, if $\ell(\lambda^+)+\ell(\lambda^-) \le n$, the irreducible representation $\GL_n(\lambda^+,\lambda^-)$ has the highest weight
\[ (\lambda^+_1L_1 + \lambda^+_2L_2 + \dots  ) - ( \lambda^-_1L_{n-1} + \lambda^-_2L_{n-2} + \dots).\]
Here $L_i \in \mathfrak h^*$ are elements of the dual vector space of the $n\times n$ diagonal matrices $\mathfrak h \cong \QQ^n$. The matrices $E_{i,i}$ sending $e_i$ to itself and all other $e_j$ to zero gives a basis of $\mathfrak h$ and \[L_i(E_{j,j}) = \delta_{i,j}\] gives its dual basis. For more details on the notation see Fulton--Harris \cite[\S15]{FH}.

Note that $ \GL_n(\lambda,\emptyset)=\GL_n(\lambda) $ is polynomial and $\GL_n(\emptyset)$ is the trivial representation.
Another significant representation is the one-dimensional \emph{determinant representation} $D$ given by
\[ g\cdot 1 =\det g\cdot 1.\]
For each $k\in \ZZ$ let $D_k$ be the one-dimensional representation given by
\[ g \cdot 1 = (\det g)^k \cdot 1.\]
The highest weight of $D_k$ is 
\[ k(L_1 +\dots + L_n).\]
Interestingly, if $V$ is the irreducible $\GL_n\QQ$--representation with highest weight
\[ \lambda_1L_1 + \dots + \lambda_nL_n\]
for some integers $\lambda_1 \ge \dots \ge \lambda_n$, then $V\otimes D_k$ is irreducible and has the highest weight
\[ (\lambda_1+k)L_1 + \dots + (\lambda_n+k)L_n.\]



\subsection{The representation theory of $\Sp_n\QQ$}
For the symplectic groups every rational representation is already polynomial. As for the general linear groups, the rational representation theory of $\Sp_{2n}\QQ$ is semisimple and every irreducible representation is finite dimensional. 

The standard representation $V_n=\QQ ^{2n}$ is irreducible.  All other irreducible {rational representations} can be constructed as subquotients of the $r$-fold tensor product $V_n^{\otimes r}$ for some $r\in\NN$ on which $\Sp_{2n}\QQ$ acts diagonally. Let $\langle \,\,,\,\rangle$ denote the symplectic form on $V_n$. Then for $r\ge2$ there are contractions
\begin{align*}
V_n^{\otimes r} &\longrightarrow V_n^{\otimes r-2}\\
v_1\otimes \dots \otimes v_r & \longmapsto \langle v_i,v_j\rangle \cdot v_1 \otimes \dots \otimes \hat v_i \otimes \dots \otimes \hat v_j \otimes \dots \otimes v_r.
\end{align*}
Let $V_n^{\langle r\rangle}$ denote the intersection of the kernels of all these maps. Then
\[ \Sp_{2n}(\lambda) := V_n^{\langle r \rangle}  \tens[\QQ \mathfrak S_r] \mathfrak S_r(\lambda)  \]
is a rational $\Sp_{2n}\QQ$--representation. It is irreducible if $\lambda$ has at most $n$ rows and zero otherwise. All irreducible rational $\Sp_{2n}\QQ$--representations are of this form. 


In terms of weights, if $\ell(\lambda) \le n$, the irreducible representation $\Sp_{2n}(\lambda)$ has the highest weight
\[ \lambda_1L_1 + \lambda_2L_2 + \dots + \lambda_nL_n .\]
Here $L_i \in \mathfrak h^*$ are elements of the dual vector space of the Cartan subalgebra $\mathfrak h \cong \QQ^n$ of the $2n\times 2n$ matrices generated by the basis $H_i = (E_{2i-1,2i-1} - E_{2i,2i})$ with $i=1, \dots, n$. Then $L_i$ is the dual basis with \[L_i(H_j) = \delta_{i,j}.\]  For more details on the notation see \cite[\S16+\S17]{FH}.

\subsection{Littlewood--Richardson coefficients}

Subsequently, we will make extensive use of the Littlewood--Richardson coefficients $c^\lambda_{\mu\nu}$. These arise in various situations, especially in the context of branching rules, which we wish to cover in the next subsections. An introduction to these coefficients can be found in Fulton \cite{Youngtableaux}. For our purpose the following two propositions suffice. 

The first proposition showcases the role of the Littlewood--Richardson coefficients in branching rules. Throughout the paper we use the abbreviation
\[ [V,W] = \dim \Hom_G(V,W)\]
for $G$--representations $V$ and $W$. If $V$ is simple and $W$ semisimple, $[V,W]=[W,V]$ counts the multiplicity of $V$ in $W$.

\begin{prop}
Let $\lambda, \mu, \nu$ be partitions. Then the Littlewood--Richardson coefficient $c^\lambda_{\mu\nu}$ computes the following multiplicities.
\begin{equation*}
[ \Res^{\mathfrak S_{m+n}}_{\mathfrak S_{m}\times\mathfrak S_n} \mathfrak S_{m+n}(\lambda), \mathfrak S_m(\mu) \otimes \mathfrak S_n(\nu) ] = c^\lambda_{\mu\nu}
\end{equation*}
if $|\lambda| = m+n$, $|\mu|= m$, $|\nu|=n$.
\begin{equation*}
[ \Res^{\GL_{m+n}\QQ}_{\GL_{m}\QQ\times\GL_n\QQ} \GL_{m+n}(\lambda), \GL_m(\mu) \otimes \GL_n(\nu) ] = c^\lambda_{\mu\nu}
\end{equation*}
if $\ell(\lambda) \le m+n$, $\ell(\mu) \le m$, $\ell(\nu)\le n$.
\begin{equation*}
[ \GL_n(\mu)\otimes \GL_n(\nu) , \GL_n(\lambda)] = c^{\lambda}_{\mu\nu} 
\end{equation*}
if $\ell(\lambda), \ell(\mu), \ell(\nu) \le n$.
\end{prop}

The second proposition implies that all sums over partitions that appear in this paper are finite sums.

\begin{prop}
The Littlewood--Richardson coefficient $c^\lambda_{\mu\nu}$ is zero unless
\[ |\mu| + |\nu| = |\lambda|\]
and both $\mu$ and $\nu$ are subdiagrams of $\lambda$.
\end{prop}

\subsection{Some simple branching rules}

The main tool of this paper will be the branching rules for rational representations. For the restrictions $\Res^{\GL_n\QQ}_{\GL_{n-1}\QQ} \GL_n(\lambda^+,\lambda^-)$ and $\Res^{\Sp_{2n}\QQ}_{\Sp_{2n-2}\QQ}\Sp_{2n}(\lambda)$ there are some simple rules that can be found in Goodman--Wallach \cite[Thm 8.1.1, Thm 8.1.5]{GW09}. To phrase these for the rational representations of the general linear groups, let $\lambda \in \ZZ^n$ with
\[ \lambda_1 = \lambda^+_1\ge  \lambda_2=  \lambda^+_2 \ge \dots\ge \lambda_{n-1} = -\lambda^-_2\ge \lambda_n =- \lambda^-_1\]
for a pair of partitions $(\lambda^+,\lambda^-)$ with length $\ell(\lambda^+)+\ell(\lambda^-) \le n$.

\begin{thm}\label{thm:GLres}
The multiplicity
\[[ \Res^{\GL_n\QQ}_{\GL_{n-1}\QQ} \GL_n(\lambda^+,\lambda^-), \GL_{n-1}(\mu^+,\mu^-)] =1\] if and only if
\[ \lambda_1 \ge\mu_1\ge\lambda_2\ge \mu_2\ge \dots \ge \lambda_{n-1}\ge\mu_{n-1}\ge\lambda_n.\]
Otherwise it is zero.
\end{thm}

\begin{thm}\label{thm:Spres}
The multiplicity
\[ [ \Res^{\Sp_{2n}\QQ}_{\Sp_{2n-2}\QQ}\Sp_{2n}(\lambda), \Sp_{2n-2}(\mu)] \]
is nonzero if and only if
\[ \lambda_i\ge \mu_i \ge \lambda_{i+2}\]
for all $1\le i\le n-1$. ($\lambda_{n+1}$ is always zero.) 
\end{thm}

In both cases, we obtain a corollary, which will prove useful later.

\begin{cor}\label{cor:GLres}
If
\[[ \Res^{\GL_n\QQ}_{\GL_{n-m}\QQ} \GL_n(\lambda^+,\lambda^-), \GL_{n-m}(\mu^+,\mu^-)] \neq 0\] 
then
\[ \ell(\mu^+)+\ell(\mu^-) \ge \ell(\lambda^+)+\ell(\lambda^-)-2m.\]
\end{cor}

\begin{proof}
First assume $m=1$. Let $r^+=\ell(\lambda^+), r^-=\ell(\lambda^-)$. Then from \autoref{thm:GLres}, we know that
\[ \mu^+_{r^+-1} \ge \lambda^+_{r^+} >0 >-\lambda^-_{r^-}\ge \mu^-_{r^--1}.\]
This implies
\[ \ell(\mu^+)+ \ell(\mu^-) \ge r^+-1+ r^--1= \ell(\lambda^+)+\ell(\lambda^-)-2\]
and proves the assertion for $m=1$. For $m>1$ the corollary follows by induction.
\end{proof}

\begin{cor}\label{cor:Spres}
If
\[ [ \Res^{\Sp_{2n}\QQ}_{\Sp_{2n-2m}\QQ}\Sp_{2n}(\lambda), \Sp_{2n-2m}(\mu)]\neq 0 \]
then
\[  \ell(\mu) \ge \ell(\lambda)-2m.\]
\end{cor}

\begin{proof}
For $m=1$ this follows from \autoref{thm:Spres}, because
\[ \mu_{r-2}\ge\lambda_r>0\]
for $r= \ell(\lambda)$. For $m>1$ the corollary follows by induction.
\end{proof}

For relatively small partitions $\lambda^+,\lambda^-,\lambda$, the multiplicities of some irreducible constituents of  $\Res^{\GL_{m+n}\QQ}_{\GL_m\QQ\times\GL_{n}\QQ} \GL_n(\lambda^+,\lambda^-)$ and $\Res^{\Sp_{2m+2n}\QQ}_{\Sp_{2m}\QQ\times\Sp_{2n}\QQ}\Sp_{2n}(\lambda)$ can be expressed nicely in the stable branching rules. For the following results we quote Howe--Tan--Willenbring \cite[2.2.1, 2.2.3]{HTW}.

\begin{thm}\label{thm:GLstable}
Let $m,n,p,q\in \mathbb N$ such that $p+q\le \min(m,n)$. Let $\lambda^+,\mu^+,\nu^+$ be partitions with at most $p$ rows and  $\lambda^-,\mu^-,\nu^-$ with at most $q$ rows. Then
\begin{multline*} [\Res_{\GL_m\QQ\times\GL_n\QQ}^{\GL_{m+n}\QQ} \GL_{m+n} (\lambda^+,\lambda^-)\,,\, \GL_m(\mu^+,\mu^-) \otimes \GL_n(\nu^+,\nu^-)] \\= \sum_{\gamma^+, \gamma^-, \delta} c_{\mu^+ \nu^+}^{\gamma^+}c_{\mu^- \nu^-}^{\gamma^-}c_{\gamma^+ \delta}^{\lambda^+}c_{\gamma^- \delta}^{\lambda^-} .\end{multline*}
\end{thm}

\begin{thm}\label{thm:Spstable}
Let $\lambda,\mu,\nu$ be partitions with at most $ \min(m,n)$ rows. Then
\[ [\Res_{\Sp_{2m}\QQ\times\Sp_{2n}\QQ}^{\Sp_{2m+2n}\QQ} \Sp_{2m+2n}(\lambda)\,,\, \Sp_{2m}(\mu)\otimes \Sp_{2n}(\nu)] = \sum_{\gamma, \delta} c_{\mu \nu}^{\gamma} c_{\gamma (2\delta)'}^{\lambda} \]
where $(2\delta)'$ is a partition with only even column lengths.
\end{thm}

\subsection{Modification rules for $\GL_n\QQ$}


In order to state the branching rules more generally, we need modification rules. To that effect we will use Koike and Tereda's theory of universal characters introduced in \cite{KT,Ko}.

Let 
\[ \Lambda_x = \varprojlim_n \ZZ[x_1, \dots, x_n]^{\mathfrak S_n}\]
denote the ring of symmetric functions and
\[ \Lambda_{xy} = \Lambda_x \otimes \Lambda_y.\]
The Schur functions \[\{s_\lambda(x)\}_{\lambda \text{ a partition}}\] form a basis of the free abelian group $\Lambda_x$ and thus the tensor products \[\{s_\lambda(x)\otimes s_\mu(y)\}_{\lambda,\mu\text{ partitions}}\] form a basis of $\Lambda_{xy}$. 
Koike \cite[Sec 2]{Ko} defines a ring homomorphism
\[ \tilde\pi_n \colon \Lambda_{xy} \longrightarrow R(\GL_n\QQ)\]
to the representation ring  $R(\GL_n\QQ)$ of the rational representations of $\GL_n\QQ$. We denote
\[ \modGL_n(\lambda^+,\lambda^-) = \tilde\pi_n(s_{\lambda^+}(x)\otimes s_{\lambda^-}(y)).\]

The ``$\mathrm{mod}$'' stands for modification and and the laws that govern these ``modified representations'' are known under the name \emph{modification rules}.

The idea is that the ring structure of $\Lambda_{xy}$ controls the branching rules of the rational representations of the general linear group. By design
\[ \modGL_n(\lambda^+,\lambda^-) = \GL_n(\lambda^+,\lambda^-) \] 
if $\ell(\lambda^+)+\ell(\lambda^-) \le n$. In general, $\modGL_n(\lambda^+,\lambda^-)$ is zero or a virtual representation $\pm\GL_n(\mu^+,\mu^-)$ for some partitions $\mu^+,\mu^-$ with $\ell(\mu^+)+\ell(\mu^-) \le n$. These two statements are precisely \cite[Prop 2.2]{Ko}.

Sam--Snowden--Weyman  \cite[Sec 5.4]{SSW} give the following combinatorial construction of the modification rules. A \emph{border strip} is a skew Young diagram that does not contain a $2\times 2$ square. Its length is the number of boxes it contains. Assume $\ell(\lambda^+)+\ell(\lambda^-) > n$. Let, if they exist, $R_{\lambda^+}$ and $R_{\lambda^-}$ be the connected border strips of length $\ell(\lambda^+)+\ell(\lambda^-) - n-1$ in $\lambda^+$ and $\lambda^-$ containing the first box in the last row, respectively. If $\lambda^+\setminus R_{\lambda^+}$ and $\lambda^-\setminus R_{\lambda^-}$ are both Young diagrams again then
\[ \modGL_n(\lambda^+,\lambda^-) = (-1)^{c( R_{\lambda^+})+c( R_{\lambda^-})-1}\cdot \modGL_n(\lambda^+\setminus R_{\lambda^+},\lambda^-\setminus R_{\lambda^-}),\]
where $c(R)$ denotes the number of columns the skew diagram $R$ occupies. If $R_{\lambda^+}$  or $R_{\lambda^-}$  do not exist or are empty, or $\lambda^+\setminus R_{\lambda^+}$ or $\lambda^-\setminus R_{\lambda^-}$ are not Young diagrams then
\[ \modGL_n(\lambda^+,\lambda^-) = 0.\]
In this construction 
\[ \ell(\lambda^+\setminus R_{\lambda^+}) + \ell(\lambda^-\setminus R_{\lambda^-}) < \ell(\lambda^+)+\ell(\lambda^-).\]
Therefore it terminates after finitely many steps. 

We reproduce \cite[Ex 5.17]{SSW}. Let $n=3$, $\lambda^+ = ( 4, 3, 2, 2)$ and $\lambda^-=  (5, 2, 2, 1, 1)$. Then $\ell(\lambda^+)+\ell(\lambda^-) = 9$ and the border strips of length $9-3-1 = 5$ are marked by bullet points in following diagrams:
\[ \lambda^+ = \young(\ \ \ \ ,\ \bullet \bullet ,\ \bullet ,\bullet \bullet ) \quad\quad \lambda^- = \young(\ \ \ \ \ ,\ \bullet,\bullet \bullet,\bullet,\bullet)\]
Because $R_{\lambda^+}$ occupies $3$ columns and $R_{\lambda^-}$ occupies $4$ columns,
\[ \modGL_3\left(\tiny\yng(4,3,2,2),\tiny\yng(5,2,2,1,1)\right) = \modGL_3\left(\tiny\yng(4,1,1),\tiny\yng(5,1)\right).\]
We again mark the border strips of length $5-3-1=1$ by bullet points:
\[  \young(\ \ \ \ ,\   ,\bullet  ) \quad\quad  \young(\ \ \ \ \ ,\bullet )\]
In the end we get
\[ \modGL_3\left(\tiny\yng(4,3,2,2),\tiny\yng(5,2,2,1,1)\right) = - \modGL_3\left(\tiny\yng(4,1),\tiny\yng(5)\right) = -\GL_3\left(\tiny\yng(4,1),\tiny\yng(5)\right).\]

We summarize all that we will need in the later discussion in the following Proposition.

\begin{prop}[Modification rules for $\GL_n\QQ$]\label{prop:modGL}
Let $\lambda^+,\lambda^-$ be partitions. Then:
\begin{enumerate}
\item \label{item:stablemodGL} $\modGL_n(\lambda^+,\lambda^-) = \GL_n(\lambda^+,\lambda^-)$ if $\ell(\lambda^+)+\ell(\lambda^-) \le n$.
\item $\modGL_n(\lambda^+,\lambda^-)$ is zero or a virtual representation $\pm\GL_n(\mu^+,\mu^-)$ for some partitions $\mu^+,\mu^-$ with $\ell(\mu^+)+\ell(\mu^-) \le n$.
\item \label{item:modincludingGL}If $\modGL_n(\lambda^+,\lambda^-) = \pm\GL_n(\mu^+, \mu^-)$ then $\mu^+,\mu^-$ are contained in $\lambda^+,\lambda^-$, respectively.
\end{enumerate}
\end{prop}

\subsection{Modification rules for $\Sp_{2n}\QQ$}

Koike--Tereda \cite[Sec 2.1]{KT} denote by
\[ \{\chi_{\Sp}(\lambda)(x)\}_{\lambda\text{ a partition}}\]
another basis of $\Lambda_x$. They define in \cite[Sec 2.2]{KT} a ring homomorphism
\[ \pi_{\Sp_{2n}}\colon \Lambda_x \longrightarrow R(\Sp_{2n}\QQ)\]
to the representation ring  $R(\Sp_{2n}\QQ)$ of the rational representations of $\Sp_{2n}\QQ$. We denote 
\[ \modSp_{2n}(\lambda) = \pi_{\Sp_{2n}}(\chi_{\Sp}(\lambda)(x)).\]

Similar to the modification rules of the general linear group
\[\modSp_{2n}(\lambda) = \Sp_{2n}(\lambda)\]
if $\ell(\lambda) \le n$ and otherwise $\modSp_{2n}(\lambda)$ is zero or a virtual representation of the form $\pm\Sp_{2n}(\mu)$ for some partition $\mu$ with $\ell(\mu)\le n$. Koike--Tereda prove this in \cite[Prop 2.2.1(1)+Prop 2.4.1(ii)]{KT}.

Sam--Snowden--Weyman  \cite[Sec 3.4]{SSW} give the following combinatorial construction of the modification rules. Assume $\ell(\lambda)>n$. Let, if it exists, $R_{\lambda}$ be the connected border strip of length $2(\ell(\lambda) - n-1)$ in $\lambda$ containing the first box in the last row. If $\lambda\setminus R_\lambda$ is a Young diagram again then
\[ \modSp_{2n}(\lambda) = (-1)^{c(R_\lambda)}\cdot \modSp_{2n}(\lambda\setminus R_{\lambda}).\]
If $R_\lambda$ does not exist or is empty, or $\lambda\setminus R_\lambda$ is not a Young diagram then
\[ \modSp_{2n}(\lambda) = 0.\]
Again
\[ \ell(\lambda\setminus R_\lambda) < \ell(\lambda)\]
implies that this procedure terminates after finitely many steps.

The following example is \cite[Ex 3.20]{SSW}. Let $n=2$ and consider the partition $\lambda = (6, 5, 4, 4, 3, 3, 2)$. We give the border strips of all steps in the following picture:
\[  {{\young(\ \ \ \ \ \ ,\ \ \ \ \ ,\ \ \ \bullet ,\ \ \bullet \bullet ,\ \ \bullet ,\ \bullet \bullet ,\bullet \bullet )}\rightsquigarrow \quad {\young(\ \ \ \ \ \ ,\ \ \ \ \ ,\ \bullet \bullet  ,\ \bullet   ,\bullet \bullet  ,\bullet)}\rightsquigarrow \quad {\young(\ \ \ \ \ \ ,\ \ \ \ \ ,\bullet    ,\bullet   )}\rightsquigarrow \quad {\young(\ \ \ \ \ \ ,\ \ \ \ \        )}}\] 
Therefore
\[ \modSp_4\left(\tiny\yng(6,5,4,4,3,3,2)\right) = -\modSp_4\left(\tiny\yng(6,5,3,2,2,1)\right) = -\modSp_4\left(\tiny\yng(6,5,1,1)\right) = -\modSp_4\left(\tiny\yng(6,5)\right). \]

We summarize all that we will need in the later discussion in the following Proposition.

\begin{prop}[Modification rules for $\Sp_{2n}\QQ$]\label{prop:modSp}
Let $\lambda$ be a partition. Then:
\begin{enumerate}
\item \label{item:stablemodSp}$\modSp_{2n}(\lambda) = \Sp_{2n}(\lambda)$ if $\ell(\lambda) \le n$.
\item $\modSp_{2n}(\lambda)$ is zero or a virtual representation $\pm\Sp_{2n}(\mu)$ for some partition $\mu$ with $\ell(\mu)\le n$.
\item \label{item:modincludingSp} If $\modSp_{2n}(\lambda) = \pm \Sp_{2n}(\mu)$ then $\mu$ is contained in $\lambda$.
\end{enumerate}
\end{prop}

\subsection{Branching rules}

The main technical tool of this paper will be branching rules of rational representations. We will need formulas for inner and outer tensor products and a stability statement for plethysms. These are corollaries of the modification rules.

%
%
%

\begin{thm}[Koike {\cite[Thm 2.4]{Ko}}]\label{thm:innerGL}
Let  $\mu^+,\mu^-,$ $\nu^+,\nu^-$ with $\ell(\mu^+)+\ell(\mu^-), \ell(\nu^+)+\ell(\nu^-) \le n$. Then
\begin{multline*} \GL_n(\mu^+,\mu^-)\otimes \GL_n(\nu^+,\nu^-) \cong\\  \bigoplus_{\lambda^+, \lambda^-} \modGL_{n}(\lambda^+, \lambda^-)^{\oplus \sum\limits_{\alpha^+,\alpha^-, \beta^+, \beta^-, \gamma, \delta}c_{\alpha^+\beta^+}^{\lambda^+}c_{\alpha^+\gamma}^{\mu^+}c_{\beta^+\delta}^{\nu^+}c_{\alpha^-\beta^-}^{\lambda^-}c_{\alpha^-\delta}^{\mu^-}c_{\beta^-\gamma}^{\nu^-}}.   \end{multline*}
\end{thm}

\begin{thm}[Koike {\cite[Thm 3.1]{Ko}}]\label{thm:innerSp}
Let  $\mu,\nu$ with $\ell(\mu), \ell(\nu) \le n$. Then
\[ \Sp_{2n}(\mu)\otimes \Sp_{2n}(\nu) \cong  \bigoplus_{\lambda} \modSp_{2n}(\lambda)^{\oplus \sum\limits_{\alpha, \beta, \gamma}c_{\alpha\beta}^{\lambda}c_{\alpha\gamma}^{\mu}c_{\beta\gamma}^{\nu}}.  \]
\end{thm}

\begin{thm}[Koike {\cite[Prop 2.6]{Ko}}]\label{thm:outerGL}
Let  $\lambda^+,\lambda^-$ with $\ell(\lambda^+)+\ell(\lambda^-) \le m+n$. Then
\begin{multline*} \Res_{\GL_m\QQ\times\GL_n\QQ}^{\GL_{m+n}\QQ} \GL_{m+n} (\lambda^+,\lambda^-) \cong\\  \bigoplus_{\mu^+,\mu^-,\nu^+,\nu^-} \Big(\modGL_{m}(\mu^+,\mu^-) \otimes \modGL_{n}(\nu^+,\nu^-)\Big)^{\oplus \sum\limits_{\gamma^+, \gamma^-, \delta}c_{\mu^+\delta}^{\gamma^+}c_{\mu^- \delta}^{\gamma^-}c_{\gamma^+ \nu^+}^{\lambda^+}c_{\gamma^- \nu^+}^{\lambda^-}}.   \end{multline*}
\end{thm}

\begin{thm}\label{thm:outerSp}
Let $\lambda$ be a partition with $\ell(\lambda)\le m+n$. Then
\[ \Res_{\Sp_{2m}\QQ\times\Sp_{2n}\QQ}^{\Sp_{2m+2n}\QQ} \Sp_{2m+2n} (\lambda) \cong  \bigoplus_{\mu,\nu} \Big(\modSp_{2m}(\mu) \otimes \modSp_{2n}(\nu)\Big)^{\oplus \sum\limits_{\gamma, \delta} c_{\mu\nu}^{\gamma}c_{\gamma (2\delta)'}^{\lambda}}.   \]
\end{thm}

\begin{proof}
In the philosophy of the proof of \cite[Prop 2.6]{Ko}, we consider two variable sets $x,y$ and the natural embedding 
\[ \Lambda_{x\cup y} \longrightarrow \Lambda_x \otimes \Lambda_y\]
where $x\cup y$ is the union of the variable sets $x$ and $y$. Then there is a unique way to write
\[ \chi_{\Sp}(\lambda)(x\cup y) = \sum m^\lambda_{\mu \nu} \chi_{\Sp}(\mu)(x) \otimes \chi_{\Sp}(\nu)(y).\]
Let $N\ge \max(\ell(\lambda), \ell(\mu)+\ell(\nu))$. Consider the following commutative diagram.
\[\xymatrix{
\Lambda_{x\cup y} \ar[r]\ar[d] & \Lambda_x \otimes \Lambda_y\ar[d]\\
R(\Sp_{4N}\QQ) \ar[r] & R(\Sp_{2N}\QQ) \otimes R(\Sp_{2N}\QQ)
}\]
Then the stable branching rule \autoref{thm:Spstable} and  \hyperref[prop:modSp]{\autoref{prop:modSp}(\ref{item:stablemodSp})} imply
\[ m^\lambda_{\mu\nu}= \sum\limits_{\gamma, \delta} c_{\mu\nu}^{\gamma}c_{\gamma (2\delta)'}^{\lambda}.\]
Then the following commutative diagram proves the assertion.
\[
\begin{gathered}[b]
\xymatrix{
\Lambda_{x\cup y} \ar[r]\ar[d] & \Lambda_x \otimes \Lambda_y\ar[d]\\
R(\Sp_{2m+2n}\QQ) \ar[r] & R(\Sp_{2m}\QQ) \otimes R(\Sp_{2n}\QQ)
}\\[-\dp\strutbox]
\end{gathered}\qedhere\]
\end{proof}

In the last paragraph of \cite[Sec 2]{Ko} \emph{plethysms} in the universal character ring $\Lambda_{xy}$ are introduced. That is if $\lambda,\mu^+,\mu^-$ are partitions then there is an element $s_\lambda(x) \circ (s_{\mu^+}(x) \otimes s_{\mu^-}(y)) \in \Lambda_{xy}$ such that
\[ \tilde \pi_n( s_\lambda(x) \circ (s_{\mu^+} (x)\otimes s_{\mu^-}(y))) = \GL_n(\mu^+,\mu^-)^{\otimes |\lambda|} \tens[\QQ\mathfrak S_{|\lambda|}] \mathfrak S_{|\lambda|}(\lambda)\]
for all $n\ge \ell(\mu^+)+\ell(\mu^-)$. We need the following consequence.

\begin{prop}\label{prop:wedgeGL}
Let $\lambda^+,\lambda^-$ be partitions and $k\in \NN$ then there is a large $N\in\NN$ and fixed coefficients $m_{\mu^+\mu^-}$ such that 
\[ {\bigwedge}^k \GL_n(\lambda^+,\lambda^-) \cong \bigoplus_{\mu^+,\mu^-} \GL_n(\mu^+,\mu^-)^{\oplus m_{\mu^+\mu^-}}\]
for all $n\ge N$.
\end{prop}

\begin{proof}
Let us write
\[ s_{(1^k)}(x)\circ (s_{\lambda^+}(x)\otimes s_{\lambda^-}(y)) = \sum_{\mu^+,\mu^-} m_{\mu^+\mu^-}\cdot s_{\mu^+}(x)\otimes s_{\mu^-}(y)\]
in $\Lambda_{xy}$. This is a finite sum and let $N$ be the maximal value $\ell(\mu^+)+\ell(\mu^-)$ of those pairs $(\mu^+,\mu^-)$ for which $m_{\mu^+\mu^-} \neq 0$. Then applying $\tilde\pi_n$ gives
\[ {\bigwedge}^k \GL_n(\lambda^+,\lambda^-) \cong \bigoplus_{\mu^+,\mu^-} \GL_n(\mu^+,\mu^-)^{\oplus m_{\mu^+\mu^-}}\]
for all $n\ge N$ as asserted.
\end{proof}

The analogous statement for $\Sp_{2n}\QQ$ can be found in \cite{CF} or can be proved analogously.

\begin{prop}[Church--Farb {\cite[Thm 3.1]{CF}}]\label{prop:wedgeSp}
Let $\lambda$ be a partition and $k\in \NN$ then there is a large $N\in \NN$  and  fixed coefficients $m_{\mu}$ such that 
\[ {\bigwedge}^k \Sp_n(\lambda) \cong \bigoplus_{\mu} \Sp_n(\mu)^{\oplus m_{\mu}}\]
for all $n\ge N$.
\end{prop}

%

The following corollaries are needed in  \autoref{section:fg}.


\begin{cor}\label{cor:branchingGL}
Let $\lambda^+,\lambda^-$ be partitions with $\ell(\lambda^+) +  \ell(\lambda^-) \le n$ and let $\mu^+, \mu^-$ be partitions with $\ell(\mu^+) +  \ell(\mu^-) \le m$. Assume further $|\lambda^+|+|\lambda^-| \le |\mu^+|+|\mu^-|$, then
\begin{multline*} [\Res^{\GL_n\QQ}_{\GL_m\QQ\times \GL_{n-m}\QQ} \GL_n(\lambda^+,\lambda^-), \GL_m(\mu^+,\mu^-)\otimes \GL_{n-m}(\nu^+,\nu^-)] \\= \begin{cases} 1&\text{if $\mu^+= \lambda^+$, $\mu^-= \lambda^-$ and $\nu^+=\nu^- = \emptyset$}\\ 0&\text{otherwise.}\end{cases}\end{multline*}
In particular, if $|\lambda^+|+|\lambda^-| < |\mu^+|+|\mu^-|$ then \[\Hom_{\GL_{m}\QQ} ( \GL_{m}(\mu^+,\mu^-),
\Res^{\GL_n\QQ}_{\GL_m\QQ} \GL_{n}(\lambda^+,\lambda^-) ) = 0.\]
Similarly if $\ell(\lambda^+) < \ell(\mu^+)$ or $ \ell(\lambda^-)<  \ell(\mu^-)$ then 
 \[\Hom_{\GL_{m}\QQ} ( \GL_{m}(\mu^+,\mu^-),\Res^{\GL_n\QQ}_{\GL_m\QQ} \GL_{n}(\lambda^+,\lambda^-) ) = 0.\]
\end{cor}

\begin{proof}
From \hyperref[prop:modGL]{\autoref{prop:modGL}(\ref{item:modincludingGL})}, we know that $|\modGL_m(\eta^+,\eta^-)|\le |\eta^+|+|\eta^-|$. Thus if $\modGL_m(\eta^+,\eta^-) = (\mu^+,\mu^-)$, we also know $|\eta^+|+|\eta^-| \ge |\lambda^+|+|\lambda^-|$. For such $\eta^+,\eta^-$ we can calculate the multiplicity from  \autoref{thm:outerGL}.
\[ c_{\eta^+ \nu^+}^{\gamma^+}c_{\eta^- \nu^-}^{\gamma^-}c_{\gamma^+ \delta}^{\lambda^+}c_{\gamma^- \delta}^{\lambda^-} = \begin{cases} 1&\text{if $\eta^+=\gamma^+ = \lambda^+$, $\eta^-= \gamma^-= \lambda^-$ and $\nu^+=\nu^-= \delta = \emptyset$}\\ 0&\text{otherwise.}\end{cases}\]
Therefore the only constituent \[\modGL_m(\eta^+,\eta^-)\otimes \modGL_{n-m}(\nu^+,\nu^-)\] in $\Res_{\GL_m\QQ\times\GL_n\QQ}^{\GL_{m+n}\QQ} \GL_{m+n}$ with $|\eta^+|+|\eta^-| \ge |\lambda^+|+|\lambda^-|$ is
\[\modGL_m(\mu^+,\mu^-)\otimes \modGL_{n-m}(\emptyset,\emptyset)  = \GL_m(\mu^+,\mu^-)\otimes\GL_{n-m}(\emptyset).\qedhere \]
\end{proof}

\begin{cor}\label{cor:branchingSp}
Let $\lambda$ be a partition with $\ell(\lambda)\le n$ and let $\mu$ be a partition with $\ell(\mu) \le m$. Assume further $|\lambda| \le |\mu|$, then
\[ [\Res^{\Sp_{2n}\QQ}_{\Sp_{2m}\QQ\times \Sp_{2n-2m}\QQ} \Sp_{2n}(\lambda), \Sp_{2m}(\mu)\otimes \Sp_{2n-2m}(\nu)] = \begin{cases} 1&\text{if $\mu= \lambda$ and $\nu= \emptyset$}\\ 0&\text{otherwise.}\end{cases}\]
In particular, if $|\lambda| < |\mu|$ then \[\Hom_{\Sp_{2m}\QQ} ( \Sp_{2m}(\mu),
\Res^{\Sp_{2n}\QQ}_{\Sp_{2m}\QQ} \Sp_{2n}(\lambda) ) = 0.\]
\end{cor}

\begin{proof}
Analogous to  \autoref{cor:branchingGL} we calculate the multiplicity from  \autoref{thm:outerSp} for $|\eta|\ge |\lambda|$.
\[ \begin{gathered}[b]c_{\eta \nu}^{\gamma}c_{\gamma (2\delta)'}^{\lambda} = \begin{cases} 1&\text{if $\eta=\gamma = \lambda$ and $\nu= \delta = \emptyset$}\\ 0&\text{otherwise.}\end{cases}\\[-\dp\strutbox]
\end{gathered}\qedhere\]
\end{proof}

\subsection{Restriction to $\GL_n\ZZ$ and $\Sp_{2n}\ZZ$}\label{sec:res}

We will later need to understand the restrictions $\Res^{\GL_n \QQ}_{\GL_n\ZZ} \GL_n(\lambda^+,\lambda^-)$ and $\Res^{\Sp_{2n}\QQ}_{\Sp_{2n}\ZZ} \Sp_n(\lambda)$. The information we need is provided by Borel in the context of his density theorem:

\begin{thm}[Borel {\cite[Prop 3.2]{Bo2}}]\label{thm:Borel}
Let $G$ be a simple non-compact connected real Lie group and $V$ a finite dimensional irreducible $G$--representation. Let $H$ be a subgroup of $G$ such that for every neighborhood $U$ of the identity in $G$ and every $g\in G$ there exists an integer $n>0$ with $g^n \in U\cdot H\cdot U$. Then $V$ is an irreducible $H$--representation.
\end{thm}

This theorem directly applies to the symplectic groups $G=\Sp_{2n}\RR$ and $H=\Sp_{2n}\ZZ$ and can also be transferred to stating that
\[ \Res^{\Sp_{2n}\QQ}_{\Sp_{2n}\ZZ} \Sp_{2n}(\lambda)\]
is an irreducible $\Sp_{2n}\ZZ$--representation for all partitions $\lambda$ of length $\ell(\lambda) \le n$. Furthermore because $H$ is Zariski dense in $G$ (which is the main result of \cite{Bo2}) all of these $\Sp_{2n}\ZZ$--representations are pairwise nonisomorphic.

To understand the situation for the general linear group, we need to take a look at the representation theory of the special linear group. Essentially, the difference between the rational representation theory of these two groups is the determinant representation, which restricts to the trivial $\SL_n\QQ$--representation. In fact, all irreducible polynomial (which is the same as rational) $\SL_n\QQ$--rep\-res\-en\-ta\-tions are given by and are uniquely (up to isomorphism) determined by
\[ \Res^{\GL_n\QQ}_{\SL_n\QQ}\GL_n(\lambda) = V_n^{\otimes r} \tens[\QQ\mathfrak S_r] \mathfrak S_r(\lambda)\]
for some partition $\lambda$ with length $\ell(\lambda) \le n-1$. In general, the restriction is not much harder. Every irreducible rational $\GL_n\QQ$--representation can be written as a tensor product
\[ \GL_n(\lambda^+,\lambda^-) \cong \GL_n(\lambda) \otimes D_k\]
for a uniquely determined partition $\lambda$ with length $\ell(\lambda) \le n-1$ and $k=\lambda^-_1\in \ZZ$ or $k=-\lambda^+_n\in \ZZ$ if $\lambda^-$ is empty.
Because
\[ \Res^{\GL_n\QQ}_{\SL_n\QQ} D_k \]
is the trivial representation, we have completely described the restriction of irreducible rational  $\GL_n\QQ$--representations to $\SL_n\QQ$.

 \autoref{thm:Borel} is now applicable to $G=\SL_n \RR$ and $H = \SL_n \ZZ$ and can be transferred to the statement that
\[ \Res^{\SL_n \QQ}_{\SL_n\ZZ} \SL_n(\lambda)\]
is an irreducible $\SL_n \ZZ$--representation for all partitions $\lambda$ of length $\ell(\lambda) \le n-1$. Again all these $\SL_n\ZZ$--representations are pairwise nonisomorphic. 

For the general linear group this implies that
\[ \Res^{\GL_n \QQ}_{\GL_n\ZZ} \GL_n(\lambda^+,\lambda^-)\]
is an irreducible $\GL_n\ZZ$--representation. Note that \[\Res^{\GL_n \QQ}_{\GL_n\ZZ} D_k \cong \Res^{\GL_n \QQ}_{\GL_n\ZZ} D_{k+2}\]
and thus all restrictions of irreducible rational $\GL_n\QQ$--representations to $\GL_n\ZZ$ have the form
\[ \Res^{\GL_n\QQ}_{\GL_n\ZZ} \big( \GL_n(\lambda) \otimes D_k\big)\]
for some partition $\lambda = (\lambda_1, \dots, \lambda_r)$ of length $r = \ell(\lambda)\le n-1$ and $k\in \{0,-1\}$. Such a $\GL_n \QQ$--representation has the highest weight
\[ (\lambda_1+k)L_1+\dots+(\lambda_r+k)L_r+kL_{r+1}+\dots+kL_n\]
so it is exactly one
\[ \GL_n(\lambda^+,\lambda^-)\]
such that $\ell(\lambda^+)\le n-1$ and $\lambda^-$ is contained in $(1^n)$.

 Assume
  \[  \Res^{\GL_n\QQ}_{\GL_n\ZZ} \big( \GL_n(\lambda) \otimes D_k\big) \cong \Res^{\GL_n\QQ}_{\GL_n\ZZ} \big( \GL_n(\lambda') \otimes D_{k'}\big)\]
 are isomorphic, then by restriction to $\SL_n\ZZ$, we see that $\lambda = \lambda'$. By an argument communicated to the author by David Speyer we can prove that
\[ \Res^{\GL_n\QQ}_{\GL_n\ZZ} \big( \GL_n(\lambda) \otimes D_{-1}\big) \not\cong \Res^{\GL_n\QQ}_{\GL_n\ZZ}  \GL_n(\lambda).\]
The argument goes as follows. Let $\rho$ denote the representation of $\GL_n\ZZ$ on $\GL_n(\lambda)$. If we assume there exists an isomorphism then its characters must coincide:
\[ (\det g)^{-1}\cdot \tr \rho(g)  = \tr \rho(g) \]
Thus all $g\in \GL_n \ZZ$ with negative determinant must have vanishing value of the character of $\rho$. The character can be described by the (complex) eigenvalues $\alpha_1, \dots, \alpha_n$ of $g$ as
\[ \tr \rho(g) = s_\lambda(\alpha_1, \dots, \alpha_n)\]
where $s_\lambda$ is the Schur polynomial, which is symmetric and homogeneous of degree $|\lambda|$. Therefore $s_\lambda$ can not be divisible by the inhomogeneous polynomial $1 + x_1 \cdots x_n$. Let us write $s_\lambda$ as a polynomial $p$ in the elementary  symmetric polynomials $e_1, \dots, e_n$. Then the previous statement is precisely that $p$ is not divisible by $1+e_n$. Because $\ZZ^n$ is Zariski closed in $\QQ^n$, there are integers $(f_1, \dots, f_n)$ such that $p(f_1, \dots, f_n) \neq 0$ but $1+f_n = 0$. Let
\[ g = \begin{pmatrix} 0 & \cdots &0&(-1)^{n+1}f_n\\ 1  & \ddots &\vdots &(-1)^{n} f_{n-1}\\  &\ddots&0&\vdots \\ 0 &  &1& f_1 \end{pmatrix}\in \GL_n\ZZ\]
be the companion matrix to the characteristic polynomial
\[ x^n- f_1x^{n-1} + \dots + (-1)^n f_n. \]
Say $\alpha_1, \dots, \alpha_n$ are the (complex) roots of the characteristic polynomial, that are the eigenvalues of $g$, then
\[ \tr \rho(g) = s_\lambda(\alpha_1, \dots \alpha_n) = p(f_1,\dots, f_n) \neq 0\]
even though the determinant $\det g = f_n = -1$. Contradiction.



\section{Representation stability for general linear groups and symplectic groups}\label{section:rep stab classical groups}\label{section:fg}

When Church--Ellenberg--Farb \cite{CEF} study representation stable sequences of representations of the symmetric groups, they consider modules over the category $\FI$ of finite sets and injections. When we want to generalize their work to the general linear groups and symplectic groups, the obvious generalizations of $\FI$ would be $\VI$ and $\SI$, the category of finite dimensional vector spaces and injections and the category of symplectic vector spaces and (injective) isometries. For the symplectic groups this turns out to be correct but for the general linear groups we need a different notion. The following definition of $\VIC$, which stands for vector spaces with injections and complements, was related to representation stability by Putman--Sam \cite{PS}. 

\subsection{$\VIC$ and $\SI$}

\begin{Def}\label{Def:VIC}
Fix a commutative ring $R$. Let $\VIC_R$ be the category whose objects are finite rank free modules over $R$ and its morphisms are given by a monomorphism together with a free complement of the image. That is
\[ \Hom_{\VIC_R}(V,W) = \{ (f,C) \mid f\colon V \inject W, \im f \oplus C = W, C\text{ free} \}.\]
The composition is given by
\[ (g,D)\circ (f,C) = (g\circ f, D\oplus g(C)).\]

Let $\SI_R$ be the category whose objects are finitely generated symplectic free modules over $R$ and its morphisms are given by isometries. Here a free module of rank $2n$ together with a bilinear form $\langle \ ,\,\rangle_{\Sp}$ is symplectic if there is a basis $\{e_1,e'_1,\dots, e_n,e'_n\}$ such that $\langle e_i,e_j\rangle_{\Sp} = \langle e'_i,e'_j\rangle_{\Sp} =0$ and $\langle e_i,e'_j\rangle_{\Sp} = - \langle e'_i,e_j\rangle_{\Sp} = \delta_{ij}$. Isometries are always injective but not necessarily bijective.
\end{Def}

The property by which we chose $\VIC$ and $\SI$ for our purpose is pointed out by the following remark. 

\begin{rem}\label{rem:VIC(m,n)}
A skeleton of $\VIC_R$ is given by the full subcategory on the objects $\{R^n\}_{n\in\NN}$ and
\[ \Hom_{\VIC_R}(R^m,R^n) \cong \begin{cases} \GL_nR/\GL_{n-m}R &\text{ if $n\ge m$,}\\ \emptyset &\text{ otherwise.}\end{cases} \]
Composition is given by group multiplication:
\begin{align*}
\quot{\GL_nR}{\GL_{n-m}R} \times \quot{\GL_mR}{\GL_{m-l}R} &\longrightarrow \quot{\GL_nR}{\GL_{n-l}R}\\
(g\GL_{n-m}R, h\GL_{m-l}R) &\longmapsto gh\GL_{n-l}R
\end{align*}

Similarly a skeleton of $\SI_R$ is given by the full subcategory on the objects $\{ R^{2n}\}_{n\in\NN}$ and
\[ \Hom_{\SI_R}(R^{2m},R^{2n}) \cong \begin{cases} \Sp_{2n}R/\Sp_{2n-2m}R &\text{ if $n\ge m$,}\\ \emptyset &\text{ otherwise.}\end{cases} \]
Composition is also given by group multiplication.
\end{rem}

\subsection{$\VIC$-- and $\SI$--modules}

Let us fix a commutative ring $R$.

\begin{Def}\label{Def:rational}
Let $\C$ be a category, then \emph{$\C$--modules} are functors from $\C$ to the category $\xmod{\QQ}$ of vector spaces over $\QQ$. Note that for $\C=\VIC_R$ and $\C=\SI_R$, it is enough to consider the effect of such a functor $V$ on the skeleton which is indexed by the natural numbers. By these means we will write $V_n$ for image of $R^n$ or $R^{2n}$ under $V$, respectively. Furthermore we will denote the image of the standard embedding $R^n\to R^{n+1}$ and $R^{2n} \to R^{2n+2}$ by $\phi_n\colon V_n \to V_{n+1}$.

We call $\VIC_\QQ$-- and $\SI_\QQ$--modules $V$ \emph{rational}
if all group homomorphisms $\GL_n\QQ \to \GL(V_n)$ and $\Sp_{2n}\QQ \to \GL(V_n)$ are rational.
\end{Def}

%

%

We will consider representable $\C$--modules as free.

\begin{Def}
Denote the representable functors $\QQ [\Hom_{\VIC_R}(R ^m,-)]$ and $\QQ [\Hom_{\SI_R}(R ^{2m},-)]$ uniformly by $M(m)$.

We call $\VIC_R$-- and $\SI_R$--modules $V$ \emph{generated in ranks $\le m$} if there is a surjection
\[ \bigoplus_{i\in I}M(m_i)  \surject V\]
where $m_i\le m$ for all $i \in I$. (The index set $I$ is allowed to be infinite.)

We say $V$ is \emph{generated in finite rank} if it is generated in ranks $\le m$ for some $m\in \NN$.
\end{Def}

There is a good reason to consider $M(m)$ free. By the Yoneda Lemma, a homomorphism
\[ M(m) \longrightarrow V\]
for some $\C$--module $V$ is determined by the image of $\id\in M(m)_m$ in $V_m$. Also if
\[ \bigoplus_{i\in I}M(m_i)  \surject V,\]
the smallest submodule of $V$ that contains the images of $\id\in M(m_i)_{m_i}$ is $V$ itself. Thus $V$ is generated by those images, which all lie in ranks $\le m$.

We will need the following propositions later to provide sequences of representations with a functorial structure. They are the natural generalizations of Church--Ellenberg--Farb \cite[Rem 3.3.1]{CEF} to $\VIC_R$ and $\SI_R$. Randal-Williams--Wahl \cite[Prop 4.2]{RW} prove it in a more general setup. 

\begin{prop}\label{prop:VICmod}
Let $\{V_n\}_{n\in \NN}$ be a sequence of $\GL_n R$--representations and let $\phi_n \colon V_n \to V_{n+1}$ be $\GL_n R$--equivariant maps. Then $\GL_{n-m} R$ acts trivially on the image of $V_m$ in $V_n$  if and only if there is a $\VIC_R$--module $V$ with $V(R^n) = V_n$ and $\phi_n$ is the image of the standard embedding $R^n \to R^{n+1}$.
\end{prop}

\begin{prop}\label{prop:SImod}
Let $\{V_n\}_{n\in \NN}$ be a sequence of $\Sp_{2n} R$--representations and let $\phi_n \colon V_n \to V_{n+1}$ be $\Sp_{2n} R$--equivariant maps. Then $\Sp_{2n-2m} R$ acts trivially on the image of $V_m$ in $V_n$  if and only if there is a $\VIC_R$--module $V$ with $V(R^{2n}) = V_n$ and $\phi_n$ is the image of the standard embedding $R^{2n} \to R^{2n+2}$.
\end{prop}

\noindent\fbox{\parbox[b]{\textwidth}{\textit{In what follows we often want to treat $\VIC_\QQ$ and $\SI_\QQ$ uniformly. To that end we will write $\C$ instead of $\VIC_\QQ$ and $\SI_\QQ$ when we want to make a statement that is true for both categories. We will also write $G_n$ for $\GL_n\QQ$ or $\Sp_{2n}\QQ$ depending on the setting.}}}

\subsection{Stability degree}

Analogous to the approach by Church--Ellenberg--Farb \cite[Sec 3.2]{CEF} we want to introduce the stability degree of $\C$--modules. We first make the observation that there is an injection $G_a \times G_{n-a} \to G_n$ given by a block sum. Therefore we can consider the coinvariants
\[ \QQ \tens[{\QQ G_{n-a}}] \Res^{G_n}_{G_a\times G_{n-a}}V_n\]  as a $\QQ G_{a}$--module for any $\QQ G_n$--module $V_n$. Furthermore $\phi \colon V_n \to V_{n+1}$ induces a $G_a$-map
\[\xymatrix{ \displaystyle \QQ  \tens[\QQ G_{n-a}] V_n \ar[r]^<<<<{\phi_*} & \displaystyle \QQ \tens[\QQ G_{n-a}] V_{n+1} \ar@{->>}[r] & \displaystyle \QQ \tens[\QQ G_{n+1-a}] V_{n+1}}.\]

\begin{Def}\label{Def:tau}
Let $\tau_{n,a}$ be the functor \[\tau_{n,a} V_n = \QQ  \tens[{\QQ G_{n-a}}] \Res^{G_n}_{G_a\times G_{n-a}}V_n\] from $\QQ G_n$--modules to $\QQ G_a$--modules. We say a $\C$--module $V$ has \emph{injectivity degree}, \emph{surjectivity degree} or \emph{stability degree} $\le s$ if the map
\[ \xymatrix{ \tau_{a+n,a} V_{a+n} \ar[r]^<<<<{\phi_*} & \tau_{a+n+1,a} V_{a+n+1}} \]
is injective, surjective or bijective, respectively, for all nonnegative integers $a$ and all $n\ge s$.
\end{Def}

\begin{rem}\label{rem:tau}
Note that if
\[ \Res_{G_{a}\times G_{n-a}}^{G_n} V_n \cong \bigoplus W_i\otimes W'_i,\]
with simple $\QQ G_{a} \otimes \QQ G_{n-a}$--modules $W_i\otimes W'_i$, then
\[ \tau_{n,a}V_n \cong \bigoplus_{W'_i \text{ trivial} } W_i.\]
\end{rem}

The following two propositions are analogues to \cite[Lem 3.2.7]{CEF} and follow immediately from the previous remark and \hyperref[cor:branchingGL]{Corollaries \ref{cor:branchingGL}} and \ref{cor:branchingSp}, respectively.


\begin{prop}\label{prop:GLtau}
Let $\lambda^+,\lambda^-$ be partitions with $\ell(\lambda^+) +  \ell(\lambda^-) \le n$ and let $\mu^+, \mu^-$ be partitions with $\ell(\mu^+) +  \ell(\mu^-) \le m$. Assume further $|\lambda^+|+|\lambda^-| \le |\mu^+|+|\mu^-|$, then
\[ [\tau_{n,m}\GL_n(\lambda^+,\lambda^-), \GL_m(\mu^+,\mu^-)] \\= \begin{cases} 1&\text{if $\mu^+= \lambda^+$ and $\mu^-= \lambda^-$}\\ 0&\text{otherwise.}\end{cases}\]

\end{prop}


\begin{prop}\label{prop:Sptau}
Let $\lambda$ be a partition with $\ell(\lambda)\le n$ and let $\mu$ be a partition with $\ell(\mu) \le m$. Assume further $|\lambda| \le |\mu|$, then
\[ [\tau_{n,m} \Sp_{2n}(\lambda), \Sp_{2m}(\mu)] = \begin{cases} 1&\text{if $\mu= \lambda$}\\ 0&\text{otherwise.}\end{cases}\]
\end{prop}


The next proposition is the analogue of \cite[Prop 3.1.7]{CEF}. It turns out to be much more complicated than in the case of symmetric groups studied in \cite{CEF}. Later we will only need finite surjectivity degree of $M(m)$, but we give both injectivity degree and surjectivity degree for completeness sake.

\begin{prop}\label{prop:stabdegM(m)}
$M(m)$ has injectivity degree $\le0$ and surjectivity degree $\le 2m$.
\end{prop}

\begin{proof}
From \autoref{rem:VIC(m,n)} we get that
\[ M(m)_{a+n} \cong \QQ [G_{a+n}/ G_{a+n-m}] .\]
The functor $\tau_{a+n,a}$ takes coinvariants with respect to the $G_n$-action from the left, so
\[ \tau_{a+n,a} M(m)_{a+n} \cong   \QQ [G_{a+n}/ G_{a+n-m}]_{G_n} \cong   \QQ \left[ \doublequot{G_{n}}{G_{a+n}}{ G_{a+n-m}}\right].\] 
To understand the actions better, let us specify to the general linear case and let $G_{a+n} = \GL_{a+n} \QQ$ act on the $(a+n)$-dimensional vector space with the basis
\[ \QQ^{a+n} = \QQ[ e_1, \dots, e_a,e_{a+1}, \dots, e_{a+n} ].\]
Then $G_n=\GL_n \QQ$ is the subgroup acting on the subspace
\[ \QQ^n = \QQ[ e_{a+1}, \dots, e_{a+n}]\]
and fixing $e_1,\dots, e_a$. Similarly $G_{a+n-m} = \GL_{a+n-m}\QQ$ is the subgroup acting on
\[ \QQ^{a+n-m} = \QQ[ e_{m+1}, \dots, e_{a+n}]\]
and fixing $e_1,\dots, e_m$.

The map $\phi_* \colon \tau_{a+n,a} M(m)_{a+n} \to \tau_{a+n+1,a} M(m)_{a+n+1}$ from \autoref{Def:tau} is then given by
\[ \QQ \left[ \doublequot{G_{n}}{G_{a+n}}{ G_{a+n-m}}\right] \longrightarrow  \QQ \left[ \doublequot{G_{n+1}}{G_{a+n+1}}{ G_{a+n+1-m}}\right] \]
which is in fact induced by the natural map
\[ \doublequot{G_{n}}{G_{a+n}}{ G_{a+n-m}} \longrightarrow  \doublequot{G_{n+1}}{G_{a+n+1}}{ G_{a+n+1-m}} \]
on the basis. Here we think of $G_{a+n} = \GL_{a+n}\QQ$ as a subgroup of $G_{a+n+1}= \GL_{a+n+1}$ by the standard inclusion
\[ \QQ[ e_1, \dots, e_a,e_{a+1}, \dots, e_{a+n} ] \subset \QQ[ e_1, \dots, e_a,e_{a+1}, \dots, e_{a+n}, e_{a+n+1} ].\]
Hence it suffices to consider injectivity and surjectivity for the mapping between bases.

We start with injectivity. Let $g\in G_{a+n}$, $x\in G_{n+1}$, $y\in G_{a+n+1-m}$ and assume $g'=xgy \in G_{a+n}$. We want to prove that $g$ and $g'$ represent the same element in $G_{n}\backslash G_{a+n} / G_{a+n-m}$. To do so, we use the following block matrix form.
\[ \begin{pmatrix}g'&0\\0&1\end{pmatrix} = \underbrace{\begin{pmatrix}\bar x&\tilde x \\\tilde{\tilde x} &x_0\end{pmatrix}}_{x}\begin{pmatrix}g&0\\0&1\end{pmatrix} \underbrace{\begin{pmatrix}\bar y&\tilde y \\\tilde{\tilde y} &y_0\end{pmatrix}}_{y} = \begin{pmatrix}\bar xg\bar y + \tilde x\tilde{\tilde y}&\bar xg\tilde{ y}+\tilde xy_0 \\\tilde{\tilde x}g\bar y+ x_0\tilde{\tilde y} &\tilde{\tilde x}g\tilde y+x_0y_0\end{pmatrix}\]

Now let us consider $\C= \VIC_\QQ$. Assume first $\bar x$ is invertible. Then 
\[ \bar xg \tilde y + \tilde xy_0  = 0 \implies  \tilde y + (\bar xg)^{-1}\tilde xy_0 = 0 \implies  -y_0^{-1}\tilde y = (\bar xg)^{-1}\tilde x.\]
The second implication is because not both $y_0$ and $\tilde y$ can be zero as $y$ is invertible. Thus
\[ \bar x\cdot g \cdot (\bar y - y_0^{-1}\tilde y \tilde{\tilde y}) = \bar x g \bar y + (\bar x g)(\bar x g)^{-1}\tilde x \tilde{\tilde y} = g',\]
where $\bar x \in \GL_{n}\QQ$ and $(\bar y - y_0^{-1}\tilde y \tilde{\tilde y})\in \GL_{a+n-m}\QQ$. The same argument works, when $\bar y$ is invertible. So assume that both $\bar x $ and $\bar y$ are not invertible. Then 
\[ \tilde x \not\in \im \bar x\quad \text{and}\quad \tilde{\tilde y}^{T} \not\in \im \bar y^{T}.\]
Then 
\[ \bar x (g\tilde y) = -y_0 \tilde x\]
implies $y_0=0$ and $ \bar x g \tilde y = 0$ and
\[ (\tilde{\tilde x} g) \bar y = - x_0 \tilde{\tilde y}\]
implies $x_0 = 0$ and $\tilde{\tilde x} g \bar y = 0$. Also
\[ 1 =  \tilde{\tilde x}g\tilde y + x_0y_0 = \tilde{\tilde x} g \tilde y.\]
Thus
\[ (\bar x + \tilde x\tilde{\tilde x}) \cdot g \cdot (\bar y + \tilde y \tilde{\tilde y}) = \bar x g \bar y + \bar x g \tilde y \tilde{\tilde y}+\tilde x \tilde{\tilde x}g\bar y +\tilde x \tilde{\tilde x} g\tilde y \tilde{\tilde y} = g'.\]

Now consider $\C=\SI_\QQ$. Denote by
\[ \Omega_n = \begin{pmatrix} 0&1 \\ -1&0 \\ &&\ddots\\ &&&0&1 \\ &&&-1&0 \end{pmatrix}\]
 the Gram matrix of the standard symplectic form on a $2n$-dimensional vector space. Then
\[ x \cdot \begin{pmatrix} g\tilde y \\ y_0\end{pmatrix} = \begin{pmatrix} 0\\1\end{pmatrix}\]
implies that
\[  \begin{pmatrix} g\tilde y \\ y_0\end{pmatrix} = \Omega_{n+1}x^T\Omega_{n+1}^{-1}\begin{pmatrix} 0\\1\end{pmatrix} = \Omega_{n+1}\begin{pmatrix} \tilde{\tilde x}&x_0\end{pmatrix}^{T}\Omega_1^{-1} .\]
Thus
\[  y_0 = \Omega_1 x_0^T \Omega_1^{-1}, \quad g\tilde y = \Omega_n\tilde{\tilde x}^T\Omega_1^{-1} \quad\text{and} \quad \tilde{\tilde x} g = \Omega_1\tilde y^T\Omega_n^{-1}.\] 
Because $x$ and $y$ are symplectic we derive
\[ \tilde y^T \Omega_n\bar y = - y_0^T\Omega_1 \tilde{ \tilde y}, \quad \bar x \Omega_n\tilde{\tilde x}^T = -  \tilde x\Omega_1 x_0^T\quad\text{and} \quad\tilde{\tilde x} \Omega_n\tilde {\tilde x}^T = \Omega_1-x_0\Omega_1x_0^T.\]
One can check that we can find an $\alpha \in \Sp_2\QQ = \SL_2\QQ$ such that $\alpha - x_0$ is invertible. This is equivalent to $\alpha^{-1}-y_0$ being invertible. Let $\beta, \gamma\in \GL_{2n}\QQ$ such that
\[ \beta \tilde x = \tilde x (\alpha - x_0)^{-1} \quad \text{and} \quad \tilde{\tilde y} \gamma = (\alpha^{-1}-y_0)^{-1}\tilde{\tilde y}.\]
Then
\[ \tilde x + \beta \tilde x x_0 = \beta \tilde x \alpha \quad \text{and} \quad \tilde{\tilde y} + y_0\tilde{\tilde y} \gamma = \alpha^{-1}\tilde{\tilde y}\gamma.\]
We can now calculate:
\begin{align*} \left(\bar x + \beta{ \tilde x \tilde{ \tilde x}}\right)g\left(\bar y + { \tilde y \tilde{ \tilde y}}\gamma\right) =\ & \bar x g \bar y + \beta{ \tilde x \tilde{ \tilde x}g\bar y} +  {  \bar x g\tilde y \tilde{ \tilde y}}\gamma +  \beta{ \tilde x \tilde{ \tilde x}g \tilde y \tilde{\tilde y}}\gamma\\
=\ & \bar x g \bar y + \beta{ \tilde x \Omega_1^{-1} \tilde y^T \Omega_n\bar y} +  {  \bar x \Omega_n\tilde{\tilde x}^T \Omega_1^{-1} \tilde{ \tilde y}}\gamma +  \beta{ \tilde x \tilde{ \tilde x}\Omega_n\tilde{\tilde x}^T \Omega_1^{-1} \tilde{\tilde y}}\gamma\\
=\ &\bar x g \bar y - \beta{  \tilde x \Omega_1^{-1}y_0^T\Omega_1 \tilde{\tilde y}} -  {\tilde{ x} \Omega_1x_0^T\Omega_1^{-1} \tilde{ \tilde y}}\gamma +  \beta{ \tilde x (\Omega_1-x_0\Omega_1x_0^T) \Omega_1^{-1}\tilde{\tilde y}}\gamma\\
=\ &\bar x g \bar y - \beta{  \tilde x   x_0 \tilde{\tilde y}} -  {\tilde{ x} y_0  \tilde{ \tilde y}}\gamma +  \beta{ \tilde x (1-x_0y_0)\tilde{\tilde y}}\gamma\\
=\ &\bar x g \bar y +  \tilde x \tilde{\tilde y} + \beta \tilde x \tilde{\tilde y} \gamma - ( \tilde x + \beta \tilde x x_0)(\tilde{\tilde y}+y_0\tilde{\tilde y}\beta) \\
=\ &\bar x g \bar y + \tilde x \tilde{\tilde y} = g'
\end{align*}
Also:
\begin{align*}  \left(\bar x + \beta{ \tilde x \tilde{ \tilde x}}\right) \Omega_n \left(\bar x + {\beta \tilde x \tilde{ \tilde x}}\right)^T =\ & \bar x\Omega_n\bar x^T + \bar x \Omega_n \tilde{ \tilde x}^T\tilde x^T\beta^T + {\beta\tilde{  x}\tilde{\tilde x} \Omega_n \bar x^T } +\beta\tilde{  x}\tilde{\tilde x}\Omega_n\tilde{ \tilde x}^T\tilde x^T\beta^T\\
=\ & \bar x\Omega_n\bar x^T  -  \tilde x\Omega_1 x_0^T\tilde x^T\beta^T - {\beta\tilde{  x}x_0 \Omega_1 \bar x^T } +\beta\tilde{  x}(\Omega_1-x_0\Omega_1x_0^T)\tilde x^T\beta^T\\
=\ &\bar x\Omega_n\bar x^T  +\tilde x\Omega_1 \tilde x^T + \beta\tilde x \Omega_1\tilde x^T \beta^T - (\tilde x + \beta \tilde x x_0 )\Omega_1(\tilde x + \beta \tilde x x_0 )^T\\
=\ &\bar x\Omega_n\bar x^T  +\tilde x\Omega_1 \tilde x^T + \beta\tilde x \Omega_1\tilde x^T \beta^T -\beta \tilde x\alpha \Omega_1(\beta \tilde x \alpha )^T\\
=\ &\bar x\Omega_n\bar x^T + \tilde x\Omega_1 \tilde x^T = \Omega_n
\end{align*}
And analogously:
\[ \left(\bar y + { \tilde y \tilde{ \tilde y}}\gamma\right)^T\Omega_n \left(\bar y + { \tilde y \tilde{ \tilde y}}\gamma\right) = \Omega_n\]
This proves injectivity degree $\le 0$.

For surjectivity consider first $\C = \SI_\QQ$. Let $g\in \Sp_{2(a+n+1)}\QQ$. Because $2(n+1) + 2(a+n+1-m) = 2(a+n+1) + 2(n+1-m)$, the intersection $\QQ ^{2(n+1)}\cap g\QQ ^{2(a+n+1-m)}$ is at least $2(n+1-m)$-dimensional. Assume $n\ge 2m$, then $2(n+1-m) \ge n+2$. Therefore it cannot be an isotropic subspace of $\QQ ^{2(n+1)}$. In particular, there are vectors $v,v'\in \QQ ^{2(n+1)}\cap g\QQ ^{2(a+n+1-m)}$ such that $\langle v, v'\rangle = 1$. We hence may find an $h_1\in \Sp_{2(a+n+1-m)}\QQ$ that sends $(e_{a+n+1}, e'_{a+n+1})$ to $(g^{-1}v,g^{-1}v')$ and an $h_2\in \Sp_{2(n+1)}\QQ$ that sends $(v,v')$ to $(e_{a+n+1}, e'_{a+n+1})$. Then $h_2gh_1 \in \Sp_{2(a+n)}\QQ$, thus $\Sp_{2(n+1)}\QQ g\Sp_{2(a+n+1-m)}\QQ$ is the image of $\Sp_{2n}\QQ h_2gh_1\Sp_{2(a+n-m)}\QQ$ and surjectivity degree is $\le 2m$. 

And finally for $\C =\VIC_\QQ$ let $g\in \GL_{a+n+1}\QQ$ and $n \ge 2m$. We need to find a $g'\in \GL_{a+n}$ such that
\[ \GL_{n+1}\QQ \cdot g \cdot \GL_{a+n+1-m} \QQ = \GL_{n+1}\QQ \cdot g' \cdot \GL_{a+n+1-m} \QQ.\]
That means we may do matrix transformations on $g$ from the left by $\GL_{n+1}\QQ$ and from the right by $\GL_{a+n+1-m}\QQ$ to transform to $g'$. 

Because $n \ge 2m$, we have $(n+1) + (a+n+1-m) \ge (a+n+1) + m+1$. Thus the intersection $g^{-1}\QQ ^{n+1}\cap \QQ ^{a+n+1-m}$ is at least $(m+1)$-dimensional. Let $V$ be an $(m+1)$-dimensional subspace of this intersection. We can $g$ transform such that 
\[ V = \langle\underbrace{e_{a+n+1-m}, \dots, e_{a+n+1}}_{m+1} \rangle,\]
$g$ is the identity on $V$ and sends $\langle e_{m+1}, \dots, e_{a+n-m}\rangle$ to $\langle e_1 , \dots, e_{a+n-m}\rangle$. Because
\[ \big( g\langle e_1, \dots, e_m\rangle \big)^\perp \,\cap \, V\]
is not trivial, we may transform $g$ that it sends $\langle e_1, \dots, e_m\rangle$ to $\QQ ^{a+n}$. In summary we transformed $g$ to have the form
\[ \begin{pmatrix} A &B &0&0\\ C &D &0&0\\ E&0&1&0\\0&0&0&1\end{pmatrix}\]
where the matrices $A, B, C, D, E$ have the dimensions $a\times m$, \mbox{$a\times(a+n-2m)$}, \mbox{$(n-m)\times m$}, \mbox{$(n-m)\times (a+n-m)$,} $m\times m$, respectively. And therefore $g\in \GL_{a+n}\QQ$, which proves surjectivity degree $\le 2m$.
\end{proof}

 \begin{cor}
 A $\C$--module $V$ that is generated in ranks $\le m$ has surjectivity degree $\le 2m$.
 \end{cor}

 \begin{proof}
 Note that $\tau$ is right exact, thus the following commutative diagram yields the assertion.
 \[
 \begin{gathered}[b]
 \xymatrix{
 \tau_{a+n,a} \bigoplus_{i\in I}M(m_i)_{a+n} \ar@{->>}[r]\ar@{->>}[d] &   \tau_{a+n,a}V_{a+n}\ar[d]\\
 \tau_{a+n+1,a} \bigoplus_{i\in I}M(m_i)_{a+n+1} \ar@{->>}[r] &   \tau_{a+n+1,a}V_{a+n+1}
 }\\[-\dp\strutbox]
\end{gathered}\qedhere\]
 \end{proof}

%
%

\subsection{Noetherian property}

\begin{Def}
Let 
\[ \Phi_aV = \bigoplus_{n\in \NN} \tau_{a+n,a}V_{a+n}\]
be the graded module over the graded polynomial ring $\QQ[T]$. $T$ acts via
\[\phi_*\colon \tau_{a+n,a}V_{a+n} \longrightarrow \tau_{a+n+1,a}V_{a+n+1}\]
from \autoref{Def:tau}.
\end{Def}

\begin{proof}[Proof of  {\hyperref[thmA:VICnoeth]{Theorems \ref{thmA:VICnoeth}}} and \ref{thmA:SInoeth}]
Let $V$ be a rational $\SI_\QQ$--module that is finitely generated in ranks $\le a$. Then the $\QQ[T]$--module
\[ \Phi_aV = \bigoplus_{n\in\NN} \tau_{a+n,a} V_{a+n}\]
is finitely generated in degrees $\le 2a$. The submodule
\[ \Phi_aW \subset \Phi_aV\]
is a finitely generated $\QQ[T]$--module because $\QQ[T]$ is a noetherian ring. Let $x_1\in\Phi_aW_{n_1}$, \dots, $x_r\in \Phi_aW_{n_r}$ be homogeneous generators and $w_1\in W_{a+n_1}$, \dots, $w_r\in W_{a+n_r}$ their respective preimages. Denote the submodule generated by $w_1,\dots, w_r$ by $\widetilde W\subset W$. Then
\[ \Phi_aW/\widetilde W =0.\]

Let $n\ge a$. We want to conclude that $(W/\widetilde W)_n=0$ to prove the assertion. Assume otherwise that $\Sp_{2n}(\lambda)$ is an irreducible constituent of $(W/\widetilde W)_n$. Then
\[ \tau_{n,a} \Sp_{2n}(\lambda) = 0 \]
which implies that
\[ [\Res^{\Sp_{2n}\QQ}_{\Sp_{2a}\QQ\times \Sp_{2n-2a}\QQ} \Sp_{2n}(\lambda), \Sp_{2a}(\mu)\otimes \Sp_{2n-2a}(\emptyset)] =0\]
for all $\mu$. On the other hand, we know that the image of $V_a$ in $V_n$ generates $V_n$ as an $\Sp_{2n}\QQ$--representation. Because $\Sp_{2n-2a}\QQ$ acts trivially on that image, the statement is equivalent to
\[ \Ind^{\Sp_{2n}\QQ}_{\Sp_{2a}\QQ\times \Sp_{2n-2a}\QQ}  V_a\otimes \Sp_{2n-2a}(\emptyset) \surject V_n\]
being surjective. But as seen above
\[ [\Ind^{\Sp_{2n}\QQ}_{\Sp_{2a}\QQ\times \Sp_{2n-2a}\QQ}  V_a\otimes \Sp_{2n-2a}(\emptyset), \Sp_{2n}(\lambda)] =0.\]
Contradiction to $\Sp_{2n}(\lambda)$ being a constituent of $V_n$.

The argument goes through exactly the same for a $\VIC_\QQ$--module $V$ that is finitely generated in ranks $\le a$.
\end{proof}

\subsection{Representation stability}

We will prove the main technical result of this paper. The idea for the proof stems from the proof of  \cite[Prop 3.3.3]{CEF}.

\begin{lem}\label{lem:minsize}
Let $V$ be a rational $\C$--module and $s\in \mathbb N$. Then there is a submodule $W$ such that $W_n$ contains all irreducible constituents of $V_n$ that are indexed by (pairs of) partitions of size at least $s$.
\end{lem}

\begin{proof}
By \hyperref[cor:branchingGL]{Corollaries \ref{cor:branchingGL}} and \ref{cor:branchingSp} irreducible constituents of $V_m$ that are indexed by partitions of size at least $s$ only map to irreducible constituents of $V_n$ that are indexed by partitions with size at least $s$. Therefore the described $W$ is a submodule of $V$.
\end{proof}

\begin{thm}\label{thm:fg implies repstable}
Let $V$ be a rational $\VIC_\QQ$--module or $\SI_\QQ$--module that is generated in finite rank, then $V$ is multiplicity stable. 
\end{thm}

\begin{proof}
Let  $V$ be an $\SI_\QQ$--module that is generated in finite rank. Let us write
\[ V_n \cong \bigoplus \Sp_{2n}(\lambda)^{\oplus c_{\lambda,n}}.\]
We want to prove that $c_{\lambda,n}$ is independent of large $n$.

Fix a partition $\mu$ with length $m = \ell(\mu)$ and let $W$ be the submodule of all constituents of $V$ with size at least $|\mu|+1$ as described in \autoref{lem:minsize}. Then 
\[ (V/W)_n = V_n/W_n \cong \bigoplus_{|\lambda| \le |\mu|} \Sp_{2n}(\lambda)^{\oplus c_{\lambda,n}}.\]
Now we want to count the multiplicity of the constituent $\Sp_{2m}(\mu)$ in
\[ \tau_{n,m} (V/W)_n \cong \bigoplus_{|\lambda|\le |\mu|} \tau_{n,m} \Sp_{2n}(\lambda)^{\oplus c_{\lambda,n}}.\]
By \autoref{prop:Sptau}, we get the equation:
\[
 [\tau_{n,m} (V/W)_n, \Sp_{2m}(\mu)] 
 = c_{\mu,n}
\]
 Because $V$ is generated in finite rank, so is $V/W$, which therefore has finite surjectivity degree. Thus 
 \[  c_{\mu,n} =  [\tau_{n,m} (V/W)_n, \Sp_{2m}(\mu)] \]
is a sequence of decreasing cardinal numbers once $n$ is large enough. Because the cardinal numbers are well ordered (see for example H\"onig \cite{Hoe}), this sequence stabilizes. 
 
By the exact same argument we prove that if $V$ is a $\VIC_\QQ$--module that is generated in finite  rank, $[V_n, \GL_n(\mu^+,\mu^-)]$ is independent of large $n$.
\end{proof}

%
%

\begin{lem}\label{lem:fglength}
Let $V$ be a rational $\C$--module that is generated in ranks $\le m$, then for an irreducible constituent
\[ \GL_n(\lambda^+,\lambda^-) \subset V_n\]
or
\[ \Sp_n(\lambda)\subset V_n\]
the length is bounded
\[  \ell(\lambda^+)+\ell(\lambda^-)\le 2m\]
or
\[  \ell(\lambda) \le 2m,\]
respectively.
\end{lem}

\begin{proof}
We will prove this lemma in the symplectic case using \autoref{cor:Spres}. The proof for the general linear groups goes analogously using \autoref{cor:GLres}.

The image of $V_m$ generates $V_n$ as an $\Sp_{2n}\QQ$--representation, and $\Sp_{2n-2m}\QQ$ acts trivial on it. If $\Sp_{2n}(\lambda)$ is a constituent of $V_n$, there must therefore be a constituent $\Sp_{2m}(\mu)$ of $V_m$ such that
\[ \Sp_{2m}(\mu)\otimes \Sp_{2n-2m}(\emptyset)\subset \Res^{\Sp_{2n}\QQ}_{\Sp_{2m}\QQ\times \Sp_{2n-2m}\QQ} \Sp_{2n}(\lambda).\]
Thus
\[ [\Res^{\Sp_{2n}\QQ}_{ \Sp_{2n-2m}\QQ} \Sp_{2n}(\lambda), \Sp_{2n-2m}(\emptyset)] \neq 0\]
which by \autoref{cor:Spres} implies that
\[ 0 = \ell(\emptyset) \ge \ell(\lambda) -2m.\qedhere\]
\end{proof}

\begin{proof}[Proof of {\hyperref[thmA:VICrepstab]{Theorems \ref{thmA:VICrepstab}}} and \ref{thmA:SIrepstab}]
Let $V$ be finitely generated in ranks $\le m$. By \autoref{thm:fg implies repstable}, $V$ is multiplicity stable. Let $\ker \phi$ be the submodule of $V$ given by
\[ \ker\phi_n \subset V_n.\]
Then by \autoref{thmA:SInoeth}, $\ker\phi$ is also finitely generated, which implies that $\ker\phi_n=0$ for all $n\in\NN$ large enough. This is injectivity.

Surjectivity is equivalent to being generated in finite rank.

Let us finally specialize to the symplectic groups. The proof for the general linear groups is the same. We want to prove that there are only finitely many partitions $\lambda$ such that $\Sp_{2n}(\lambda)$ is a constituent of $V_n$ for some $n\in\NN$. From \autoref{lem:fglength}, we already know that $\ell(\lambda)$ must be at most $2m$. For every fixed $n\in\NN$, there are certainly only finitely many partitions $\lambda$ such that $\Sp_{2n}(\lambda)$ is a constituent of $V_n$, because the submodule $W\subset V$ defined by
\[ W_m = \begin{cases}0&m<n\\V_m&m\ge n\end{cases}\]
would not be finitely generated otherwise.
We now consider
\[ \tau_{n,2m}V_n\]
for $n\ge 2m$. For all constituents $\Sp_{2n}(\lambda)$ of $V_n$, we know from \autoref{prop:Sptau} that
\[ \Sp_{4m}(\lambda) \subset \tau_{n,2m}V_n.\]
Because $V$ has surjectivity degree $\le 2m$, we also know that all constituents of $V_n$ for $n \ge 4m$ must already be included in the finitely generated 
\[ \tau_{4m,2m}V_{4m}.\]
This finishes the proof that $V$ is uniformly representation stable.
\end{proof}

\section{N--series and their associated Lie algebras}\label{sec:N-series}

The following definitions follow Lazard \cite{La}. He defines a generalization of a central series, such that as for the lower central series, we get a graded Lie algebra structure on the associated graded of the filtration.

\begin{Def}
For a group $\Gamma$ a filtration $\nu \Gamma$
\[ \dots \le \nu_2\Gamma \le \nu_1\Gamma = \Gamma \]
is called an \emph{N--series} if  $[ \nu_i\Gamma, \nu_j\Gamma] \le \nu_{i+j}\Gamma$.
\end{Def}

\begin{Def}[Lazard {\cite[Thm I.2.1]{La}}]\label{Def:gr} The rationalized graded Lie algebra 
\[ \gr(\nu\Gamma) = \bigoplus_{i\ge 1} \gr_i(\nu\Gamma)\]
associated to an N-series $\nu$ is defined by
\[ \gr_i(\nu\Gamma) = \nu_i\Gamma/\nu_{i+1}\Gamma \tens[\ZZ] \QQ.\]
The bracket is given by the (group) commutator.
\end{Def}

\begin{rem}\label{rem:N-series}
Let $\nu$ be an $N$--series of a group $\Gamma$. Then $\Gamma$ acts via conjugation on $\nu_i\Gamma$ for every $i\in \NN$ because
\[ [g,n] \in \nu_{i+1}\Gamma \le \nu_i\Gamma\]
for $g\in \Gamma$ and $n\in \nu_i\Gamma$. The same argument shows that $\Gamma$ acts trivially on $\gr_i(\nu\Gamma)$.
\end{rem}

\begin{Def}\label{Def:lcs}
For every group $\Gamma$ its \emph{lower central series} $\gamma\Gamma$ defined by \[\gamma_1 \Gamma = \Gamma\quad\text{and}\quad\gamma_{i+1} = [\Gamma, \gamma_i\Gamma]\] is an N-series. $\gr(\gamma\Gamma)$ is sometimes called the \emph{Malcev Lie algebra} associated to $\Gamma$.
\end{Def}

\begin{ex}
Let $\Gamma=F_n$ be the free group with $n$ generators. Then its Malcev Lie algebra is the free Lie algebra $\mathcal L_n$ with $n$ generators. 
\end{ex}

\begin{Def}\label{Def:And}
For the automorphism group $\Aut(\Gamma)$ of a group $\Gamma$,  
\[ \alpha_i \Aut(\Gamma) = \ker( \Aut(\Gamma) \to \Aut(\Gamma/\gamma_{i+1}\Gamma))\]
is called the \emph{Andreadakis filtration}.
\end{Def}

Andreadakis  \cite[Thm 1.1(ii)]{An} showed that $\alpha$ is an $N$-series of $\alpha_1\Aut(\Gamma)$.

\section{Torelli subgroups of the automorphism groups of free groups}\label{section:IA}

Let $F_n$ denote the free group on $n$ generators, then its abelianization is \[\ZZ^n \cong F_n/ [F_n,F_n].\] The quotient map induces a group homomorphism 
\[ \Aut(F_n) \longrightarrow \Aut(\ZZ^n) = \GL_n(\ZZ)\]
on their automorphism groups because the commutator subgroup $[F_n,F_n]\le F_n$ is characteristic. Nielsen \cite{Ni} proved that $\Aut(F_n)$ is generated by the permutations of the generators $x_1, \dots, x_n$ and the following two automorphisms.
\[ x_i \mapsto \begin{cases} x_1^{-1} &i=1\\x_i&i\neq 1\end{cases}\quad\text{and}\quad x_i \mapsto \begin{cases} x_1x_2 &i=1\\x_i&i\neq 1\end{cases}\]
The images of these automorphisms also generate $\GL_n\ZZ$. Hence the homomorphism between the automorphism groups is surjective.

The \emph{Torelli subgroup $\IA_n$} is defined as the kernel, so we get the following short exact sequence.
\[ 1 \to \IA_n \to \Aut(F_n) \to \GL_n \ZZ \to 1\]
As for every short exact sequence, we get an outer action of $\GL_n\ZZ$ on $\IA_{n}$, ie a group homomorphism
\[  \GL_n\ZZ \longrightarrow \Out(\IA_{n}) =\Aut(\IA_n)/ \Inn(\IA_{n}).\]
This homomorphism is given as follows. Let $g\in \GL_n\ZZ$ and $\tilde g \in \Aut(F_n)$ a preimage of $g$. Then conjugation by $\tilde g$ is an automorphism of $\IA_n$. Another preimage of $g$ is $\tilde gh$ for some $h\in \IA_n$. Then conjugation by $\tilde gh$ is conjugation by $\tilde g$ composed with the inner automorphism defined by $h$.

This outer action gives rise to a $\GL_n\ZZ$--representation on the abelianization $H_1(\IA_{n};\ZZ)$ of $\IA_{n}$ because inner automorphisms act trivially. After rationalizing the $\GL_n\ZZ$--representation
\[ H_1(\IA_{n};\QQ)   \cong {\bigwedge}^2 \QQ^n \otimes \left( \QQ^n \right)^* \]
was computed for example by Kawazumi \cite[Thm 6.1]{Kaw}. It is clearly a restriction of a $\GL_n\QQ$--representation. 

Even more, for every morphism 
\[ (f,C) \in \Hom_{\VIC_\QQ}(\QQ^m,\QQ^{n})\]
we get a unique section $s\colon \QQ^n\to \QQ^m$ of $f$ such that $C=\ker s$ and therefore a well-defined map
\[ H_1(\IA_{m};\QQ) \cong{\bigwedge}^2 \QQ^m \otimes \left( \QQ^m \right)^*  \longrightarrow H_1(\IA_{n};\QQ) \cong {\bigwedge}^2 \QQ^{n} \otimes \left( \QQ^{n} \right)^*,  \]
which turns $\{H_1(\IA_{n};\QQ)\}_{n\in \NN}$ into a $\VIC_\QQ$--module. 

For a morphism 
\[ (f,C)\in \Hom_{\VIC_\ZZ}(\ZZ^m , \ZZ^{n}) \cong \GL_n\ZZ/ \GL_{n-m}\ZZ\]
this can be traced to a group homomorphism
\[\IA_{m} \to \IA_{n}\]
up to inner automorphism of $\IA_{n}$. Here is the reason. Let $K\subset \Aut(F_n)$ be the preimage of $\GL_{n-m}\ZZ$ of the composition
\[ \Aut(F_n) \surject \GL_n \ZZ \surject \Hom_{\VIC_\ZZ}(\ZZ^m,\ZZ^n) \cong \GL_n \ZZ/ \GL_{n-m} \ZZ.\]  For $(f,C)$ we find an automorphism $g \in \Aut(F_n)$, which is uniquely determined up to right multiplication by an element of $K$. Conjugating by $g$ gives an automorphism of $\IA_n$ which can be restricted to a map $\IA_m \to \IA_n$. Because $\Aut(F_{n-m})\subset K$ surjects to $\GL_{n-m}\ZZ$, we can find for every $k\in K$ an $h \in \IA_n$ such that $hk^{-1} \in \Aut(F_{n-m})$. Since $\Aut(F_{n-m})$ commutes with $\IA_m$, the conjugation by $gk$ is the same as the conjugation by $gh$ when restricted to $\IA_m$. Thus $\IA_m \to \IA_n$ is well defined up to inner automorphism of $\IA_n$.

This group monomorphism induces a map
\[ H_1(\IA_m;\ZZ) \longrightarrow H_1(\IA_n;\ZZ).\]
This map is well defined and natural because the group monomorphism is  well defined up to inner automorphisms of $\IA_n$. 

As already been pointed out in \cite[Sec 6.2]{CF}
\[ H_1(\IA_n;\QQ) \cong {\bigwedge}^2 \QQ^n \otimes \left( \QQ^n \right)^* \cong \GL_n\left(\tiny\yng(1),\emptyset\right) \oplus \GL_n\left( \tiny\yng(1,1), \tiny\yng(1) \right) \]
for all $n\ge3$ is uniformly representation stable.

\subsection{Lower central series of $\IA_n$}\label{sec:betaGL}

Let us first consider the lower central series $\gamma\IA_n$ (see \autoref{Def:lcs}) of the Torelli subgroups $\IA_n$.

\begin{prop}\label{prop:VICZ beta}
$\{\gr_i(\gamma\IA_n)\}_{n\in\NN}$ gives rise to a $\VIC_\ZZ$--module.
\end{prop}

\begin{proof}
Because $\gamma_i\IA_n$ is a characteristic subgroup of $\IA_n$ which is normal in $\Aut(F_n)$, the latter acts on $\gamma_i\IA_n$ by conjugation. For $m\le n$, the automorphism group $\Aut(F_n)$ gives a group homomorphism
\[ \gamma_i\IA_m \longrightarrow \gamma_i\IA_n\]
that descends to the quotients
\[\gr_i(\gamma\IA_m) \longrightarrow \gr_i(\gamma\IA_n).\]
Given $g\in \Aut(F_n)$ and $h\in\Aut(F_m)$ each, the composition
\[ \gr_i(\gamma\IA_l) \stackrel{h}\longrightarrow\gr_i(\gamma\IA_m) \stackrel{g}\longrightarrow \gr_i(\gamma\IA_n)\]
is given by $gh \in \Aut(F_n)$.

Clearly $\Aut(F_{n-m})\subset  \Aut(F_n)$ acts trivially on 
\[\gr_i(\gamma\IA_m) \subset \gr_i(\gamma\IA_n)\]
because is commutes with all subquotients of $\Aut(F_m)$. Furthermore, by definition
\[ [\IA_n, \gamma_i\IA_n]=\gamma_{i+1}\IA_n.\]
Thus also $\IA_n \subset \Aut(F_n)$ acts trivially on
\[\gr_i(\gamma\IA_m) \subset \gr_i(\gamma\IA_n).\]

Therefore 
\[ \quot{\Aut(F_n)}{\IA_n\cdot\Aut(F_{n-m})} \]
gives rise to a homomorphism
\[\gr_i(\gamma\IA_m) \longrightarrow \gr_i(\gamma\IA_n).\]
But this is isomorphic to
\[ \quot{\quot{\Aut(F_n)}{\IA_n}}{\quot{\big(\IA_n\cdot\Aut(F_{n-m})\big)}{\IA_n}} \cong  \quot{\GL_n\ZZ}{\GL_{n-m}\ZZ} \]
because $\IA_n \cap \Aut(F_{n-m}) = \IA_{n-m}$.

This defines a functor $V\colon \VIC_\ZZ \to \xmod\QQ$ with
\[ V_n = \gr_i(\gamma\IA_n).\qedhere\]
\end{proof}

%

\begin{prop}
Let $V$ be a rational $\VIC_\QQ$--module which is uniformly representation stable and assume $V_n$ is finite dimensional for every $n\in\NN$. Then the $k$th degrees $\mathcal L_k(V)$ of the free Lie algebra generated by $V$ is a rational $\VIC_\QQ$--module which is uniformly representation stable and $\mathcal L_k(V_n)$ is finite dimensional for all $n\in\NN$.
\end{prop}

\begin{proof}
Clearly $\mathcal L_k$ is a functor thus $\mathcal L_k(V)$ is certainly a $\VIC_\QQ$--module. For the other assertions we adopt the methods used in the proof of \cite[Thm 5.3]{CF}. 

Because the Chevalley--Eilenberg homology of a free Lie algebra $\mathcal L(V)$ is given by
\[ H_i(\mathcal L(V)) = \begin{cases} \QQ & i=0,\\ \mathcal L_1(V) = V & i=1,\\ 0 &i >1,\end{cases}\]
for every $k\ge2$ there is an exact sequence
\[ 0 \longrightarrow \left({\bigwedge}^k\mathcal  L(V)\right)_k \longrightarrow\left({\bigwedge}^{k-1}\mathcal  L(V)\right)_k  \longrightarrow \cdots \longrightarrow \left({\bigwedge}^2\mathcal  L(V)\right)_k \longrightarrow \mathcal L_k(V) \longrightarrow 0,\]
where $\left({\bigwedge}^i\mathcal  L(V)\right)_k$ is the $k$th degree part of ${\bigwedge}^i \mathcal L(V)$, which is given by all direct summands
\[ {\bigwedge}^{i_1}\mathcal L_{k_1}(V) \otimes \cdots \otimes {\bigwedge}^{i_r}\mathcal L_{k_r}(V) \]
with $k_1 < \dots <k_r$ and $\sum i_j\cdot k_j = k$ and $\sum i_j = i$.

We can use \autoref{thm:innerGL} and \autoref{prop:wedgeGL} to deduce by induction that every term except the last of the exact sequence are sequences of finite dimensional rational $\GL_n\QQ$--representations that are uniformly representation stable. This implies that the last term $\mathcal L_k(V)$ is also a sequence of finite dimensional rational $\GL_n\QQ$--representations that is uniformly representation stable.
\end{proof}

\begin{proof}[Proof of \autoref{thmA}]\hypertarget{proof:thmA}
$\gr(\gamma\IA_n)$
is generated in the first degree
\[ \gr_1(\gamma\IA_n) \cong H_1(\IA_n;\QQ).\]
Thus there is a graded epimorphism
\[ \mathcal L(H_1(\IA_n;\QQ)) \surject \gr(\gamma\IA_n)\]
from the free Lie algebra on $H_1(\IA_n;\QQ)$.

The epimorphism 
\[ \mathcal L_i(H_1(\IA_n;\QQ)) \surject \gr_i(\gamma\IA_n)\]
is a $\GL_n\ZZ$--equivariant map because it is induced by the (group) commutator, which commutes with group homomorphisms.
Then because restrictions of irreducible rational $\GL_n\QQ$--representations to $\GL_n\ZZ$ are irreducible (see \autoref{sec:res}), the quotient $\gr_i(\gamma\IA_{n})$ is also a finite dimensional rational $\GL_n\QQ$--representation.
\end{proof}

We were not able to construct a $\VIC_\QQ$--module structure on $\{\gr_i(\gamma\IA_n)\}_{n\in\NN}$. We can however find a $\VIC_\QQ$--module $V$ for every $i\in \NN$ such that $V_n \cong \gr_i(\gamma\IA_n)$ for all large enough $n\in \NN$. To do so we appeal to \autoref{prop:VICmod}.

\begin{thm}\label{thm:VICQ beta}
Fix $i\in \NN$. There is a rational $\VIC_\QQ$--module $V$ such that $V_n$ and $\gr_i(\gamma\IA_n)$ are isomorphic $\GL_n\QQ$--representations for all large enough $n\in\NN$.
\end{thm}

\begin{proof}
We have already observed that
\[ \mathcal L_i(H_1(\IA_?;\QQ)) \surject \gr_i(\gamma \IA_?)\]
is an epimorphism of $\VIC_\ZZ$--modules. Also $\mathcal L_i(H_1(\IA_?;\QQ))$ is a uniformly representation stable $\VIC_\QQ$--module. Let $N^+$ be the maximum of all values $\ell(\lambda^+)$ such that $\GL_n(\lambda^+,\lambda^-)$ (for some $\lambda^-$) appears as a constituent in $\mathcal L_i(H_1(\IA_n;\QQ))$ for some $n\in\NN$. Similarly let $N^-$ be the maximum of all $\lambda^-$ for which a $\GL_n(\lambda^+,\lambda^-)$ is a constituent in $\mathcal L_i(H_1(\IA_n;\QQ))$ for some $n\in\NN$. Then for all $n\ge N=N^++N^-+1$ by the analysis of \autoref{sec:res} two nonisomorphic irreducible constituent of the $\GL_n\QQ$--representation $\mathcal L_i(H_1(\IA_n;\QQ))$ cannot restrict to isomorphic $\GL_n\ZZ$--representations. This means for $n\ge N$, there is a unique way to extend $\gr_i(\gamma\IA_n)$ to a rational $\GL_n\QQ$--representation such that
\[ \mathcal L_i(H_1(\IA_n;\QQ)) \surject \gr_i(\gamma \IA_n)\]
is $\GL_n\QQ$--equivariant.

Let
\[ V_n = \begin{cases} 0 &n < N\\  \gr_i(\gamma \IA_n) &n\ge N\end{cases}\]
be a sequence of $\GL_n\QQ$--representations and let
\[ \phi_n \colon V_n \longrightarrow V_{n+1}\]
be the image of the standard embedding $\ZZ^n\to \ZZ^{n+1}$ if $n\ge N$ and zero otherwise. $V_{n+1}$ only has irreducible constituents $\GL_{n+1}(\lambda^+,\lambda^-)$ with $\ell(\lambda^+)\le N^+$ and $\ell(\lambda^-)\le  N^-$ and by \autoref{cor:branchingGL}  the restriction $\Res^{\GL_{n+1}\QQ}_{\GL_n\QQ} V_{n+1}$ has therefore also only constituents $\GL_n(\mu^+,\mu^-)$ with $\ell(\mu^+)\le N^+$ and $\ell(\mu^-) \le N^-$. Therefore $\phi_n$ is $\GL_n\QQ$--equivariant.

Using \autoref{prop:VICmod}, it remains to show that $\GL_{n-m}\QQ$ acts trivially on the image of $V_m$ in $V_n$. This property is transferred from the $\VIC_\QQ$--module $\mathcal L_i(H_1(\IA_?;\QQ))$ by the following argument. We already know that the image of $\mathcal L_i(H_1(\IA_m;\QQ))$ is inside
\[ \mathcal L_i(H_1(\IA_n;\QQ))^{\GL_{n-m}\QQ}\]
which maps to
\[ V_n^{\GL_{n-m}\QQ}.\]
Because
\[ \mathcal L_i(H_1(\IA_m;\QQ)) \surject V_m\]
is surjective, $V_m$, too, must map to $ V_n^{\GL_{n-m}\QQ}$.
%
%
%
%
\end{proof}

\begin{proof}[Proof of \autoref{thmB}]
Let $V$ be the $\VIC_\QQ$--module from \autoref{thm:VICQ beta} such that $V_n = \gr_i(\gamma\IA_{n})$ for all large enough $n\in\NN$. Then by its description as a $\VIC_\ZZ$--module it restricts to the truncation of the  $\FI$--module described in \cite[Ex 7.3.6]{CEF}.

We will use the result \cite[Thm 7.3.8]{CEF} that $V$ is (the submodule of) a finitely generated $\FI$--module and thereby a finitely generated $\VIC_\QQ$--module. (Djament \cite[Prop 7.2]{Dj} proves that $\gr_i(\gamma \IA_?)$ is a $\VIC_\ZZ$--module generated in finite rank.)


Our \autoref{thmA:VICrepstab} implies then that $V$ is uniformly representation stable.
%
\end{proof}

\subsection{Johnson filtration of $\IA_n$}

Let us consider the Andreadakis filtration $\alpha\Aut(F_n)$ (see \autoref{Def:And}) of the automorphism group $\Aut(F_n)$. This is an $N$--series of
\[ \alpha_1\Aut(F_n) = \ker(\Aut(F_n) \to \Aut(F_n/\gamma_2F_n)) = \IA_n\]
and is often called the \emph{Johnson filtration} $\alpha\IA_n$ of $\IA_n$ because of Johnson's work on the Torelli subgroups of the mapping class groups of surfaces that started out with \cite{Jo80}.

\begin{prop}\label{prop:VICZ alpha}
$\{\gr_i(\alpha\IA_n)\}_{n\in\NN}$ gives rise to a $\VIC_\ZZ$--module.
\end{prop}

\begin{proof}
We follow the same strategy as in the proof of \autoref{prop:VICZ beta}. $\Aut(F_n)$ acts on its normal subgroup $\alpha_i\IA_n$ by conjugation. For $m\le n$, this action induces group homomorphisms
\[ \alpha_i\IA_m \longrightarrow \alpha_i\IA_n\]
and
\[ \gr_i(\alpha\IA_m) \longrightarrow \gr_i(\alpha\IA_n).\]

It is clear that $\Aut(F_{n-m})\subset \Aut(F_n)$ acts trivially on
\[ \gr_i(\alpha\IA_m) \subset \gr_i(\alpha\IA_n).\]
Further 
\[ [\IA_n,\alpha_i\IA_n]\subset \alpha_{i+1}\IA_n\]
follows from \cite[Thm 1.1(ii)]{An}. Thus also $\IA_n\subset \Aut(F_n)$ acts trivially on
\[ \gr_i(\alpha\IA_m) \subset \gr_i(\alpha\IA_n).\]

By same argument as in the proof of \autoref{prop:VICZ beta}, this construction gives a functor $V\colon \VIC_\ZZ \to \xmod\QQ$ with
\[ V_n = \gr_i(\alpha\IA_n).\qedhere\]
\end{proof}

In the case of the Johnson filtration we do not know whether $\gr(\alpha\IA_n)$ is generated in degree one as a Lie algebra. (Although it was conjectured by Andreadakis that $\alpha_i\IA_n = \gamma_i\IA_n$ for all $i,n \in \NN$.) Luckily we have another tool at hand---the Johnson homomorphism. As explained by Satoh  \cite[Sec 3.4]{SatSurvey} there is a $\GL_n\ZZ$--equivariant monomorphism
\[ \gr_i(\alpha\IA_n) \inject \Hom_\QQ(H_1(F_n;\QQ), \gr_{i+1}(\gamma F_n))\cong (\QQ^n)^* \otimes \mathcal L_{i+1}(\QQ^n).\]
By the same arguments used for \autoref{thmA}, we can infer the following proposition that also has been stated in the introduction of \cite[Sec 4]{SatSurvey} without a proof.

\begin{prop}\label{prop:GLQ alpha}
The natural $\GL_n\ZZ$--representation on $\gr_i(\alpha\IA_n)$ extends to a rational $\GL_n\QQ$--representation.
\end{prop}

As for the lower central series we can combine \autoref{prop:VICZ alpha} and \autoref{prop:GLQ alpha} to get the following theorem.

\begin{thm}\label{thm:alphaVICQ}
Fix $i\in \NN$. There is a rational $\VIC_\QQ$--module $V$ such that $V_n$ and $\gr_i(\alpha\IA_n)$ are isomorphic $\GL_n\QQ$--representations for all large enough $n\in\NN$.
\end{thm}

\begin{proof}[Proof of \autoref{thmD}]
Let $V$ be the $\VIC_\QQ$--module from \autoref{thm:alphaVICQ} such that $V_n = \gr_k(\alpha\IA_{n})$ for all large enough $n\in\NN$.

Church and Putman \cite{CP}  consider the groups $\IA_n$ as an $\FI$--group. They apply their \cite[Thm G]{CP} to prove their \cite[Thm C]{CP}. In the proof of the former theorem in \cite[Claim 2]{CP} it is stated that $W(k)$ is boundedly generated. But $W(k)_n\otimes_\ZZ \QQ$ is $\gr_k(\alpha\IA_n) = V_n$ for all large enough $n\in\NN$. Because every $V_n$ is finite dimensional (see \cite[Prop 3.2]{CP}), $V$ is a finitely generated $\FI$--module and thus certainly a finitely generated $\VIC_\QQ$--module.  (Djament \cite[Prop 7.3]{Dj} proves that $\gr_i(\alpha \IA_?)$ is a $\VIC_\ZZ$--module generated in finite rank independently.)


Our \autoref{thmA:VICrepstab} implies then that $V$ is uniformly representation stable. 
\end{proof}


\section{Torelli subgroups of the mapping class groups of surfaces}\label{section:I}

Let $\Sigma_{g,1}$ denote the compact, oriented genus $g$ surface with one boundary component. The mapping class group $\Mod(\Sigma_{g,1})$ is defined as the discrete group $\pi_0 \Homeo^+(\Sigma_{g,1},\partial\Sigma_{g,1})$ of isotopy classes of orientation-preserving homeomorphisms of $\Sigma_{g,1}$ that fix the boundary pointwise. 
The action of $\Mod(\Sigma_{g,1})$ on $H_1(\Sigma_{g,1};\ZZ)\cong \ZZ^{2g}$ is symplectic and the \emph{Torelli subgroup} $\I_{g,1}$ is defined to be the kernel of this action. In fact, there is a short exact sequence
\[ 1 \to \I_{g,1} \to \Mod(\Sigma_{g,1}) \to \Sp(H_1(\Sigma_{g,1};\ZZ)) \cong \Sp_{2g}\ZZ \to 1.\]
Thus we get an $\Sp_{2g}\ZZ$--representation on the abelianzation $H_1(\I_{g,1};\ZZ)$ of $\I_{g,1}$, which after rationalizing can be seen to be a restriction of a $\Sp(H_1(\Sigma_{g,1};\QQ)) \cong \Sp_{2g}\QQ$--representation
\[ H_1(\I_{g,1};\QQ) \cong {\bigwedge}^3 H_1(\Sigma_{g,1};\ZZ) \otimes \QQ ={\bigwedge}^3 H_1(\Sigma_{g,1};\QQ) \]
as it has been computed by Johnson  \cite[Thm 3(a)]{Jo}.
And for every isometry 
\[ H_1(\Sigma_{g,1};\QQ) \longrightarrow H_1(\Sigma_{g',1};\QQ)\]
there is map
\[ H_1(\I_{g,1};\QQ) \cong {\bigwedge}^3 H_1(\Sigma_{g,1};\QQ)  \longrightarrow H_1(\I_{g',1};\QQ) \cong {\bigwedge}^3 H_1(\Sigma_{g',1};\QQ),  \]
which turns $\{H_1(\I_{g,1};\QQ)\}_{g\in \NN}$ into an $\SI$--module.


As it has already been pointed out by in \cite[Sec 6.1]{CF}
\[ H_1(\I_{g,1};\QQ) \cong {\bigwedge}^3 H_1(\Sigma_{g,1};\QQ) \cong \Sp_{2g}\left(\tiny\yng(1,1,1)\right)\oplus \Sp_{2g}\left(\tiny\yng(1)\right)\]
for all $g\ge3$ is uniformly representation stable.

We will consider two N-series of $\I_{g,1}$. Denote the lower central series by $\gamma\I_{g,1}$. To construct the Johnson filtration of $\I_{g,1}$, consider the classical inclusion 
\[ \Mod(\Sigma_{g,1}) \inject \Aut(F_{2g})\]
and define
\[ \alpha_i\I_{g,1} = \Mod(\Sigma_{g,1}) \cap \alpha_i\IA_{2g} = \ker( \Mod(\Sigma_{g,1}) \to \Aut(F_{2g}/\gamma_{i+1}F_{2g})).\]
This construction immediately implies that $\alpha\I_{g,1}$ is an N-series of $\I_{g,1}$ because $\alpha\IA_{2g}$ is an N-series of $\IA_{2g}$.

\subsection{Lower central series of $\I_{g,1}$}

Analogous to \autoref{sec:betaGL} we derive the following results.

\begin{prop}\label{prop:SIZ beta}
$\{\gr_i(\gamma\I_{g,1})\}_{g\in \NN}$ gives rise to an $\SI_\ZZ$--module.
\end{prop}

\begin{prop}
Let $V$ be a rational $\SI_\QQ$--module which is uniformly representation stable and assume $V_n$ is finite dimensional for every $n\in\NN$. Then the $k$th degrees $\mathcal L_k(V)$ of the free Lie algebra generated by $V$ is a rational $\SI_\QQ$--module which is uniformly representation stable and $\mathcal L_k(V_n)$ is finite dimensional for all $n\in\NN$.
\end{prop}

Note that the following result can also be derived from the explicit description of Habegger--Sorger \cite[Thm 1.1]{HS}.

\begin{thm}\label{thm:SIQ beta}
$\{\gr_i(\gamma\I_{g,1})\}_{g\in \NN}$ gives rise to a rational $\SI_\QQ$--module.
\end{thm}

\begin{proof}
Again we have an epimorphism
\[ \mathcal L_i(H_1(\I_{?,1};\QQ)) \surject \gr_i(\gamma\I_{?,1})\]
of $\SI_\ZZ$--modules. Because $\mathcal L_i(H_1(\I_{?,1};\QQ))$ is also a rational $\SI_\QQ$--module, we get a unique rational $\Sp_{2g}\QQ$--representation structure on $\gr_i(\gamma\I_{g,1})$ that restricts to the given $\Sp_{2g}\ZZ$--representation. Therefore
\[ \mathcal L_i(H_1(\I_{g,1};\QQ)) \surject \gr_i(\gamma\I_{g,1})\]
is $\Sp_{2g}\QQ$--equivariant. We can then as in the proof of \autoref{thm:VICQ beta} lift the $\SI_\ZZ$--module structure to an $\SI_\QQ$--module structure.
\end{proof}

\begin{proof}[Proof of \autoref{thmC}]
Let $V$ be the $\SI_\QQ$--module from \autoref{thm:SIQ beta} such that $V_g = \gr_i(\I_{g,1})$. Then by its description as an $\SI_\ZZ$--module it restricts to the $\FI$--module described in \cite[Ex 7.3.6]{CEF}.

We will use the result \cite[Thm 7.3.7]{CEF} that $V$ is a finitely generated $\FI$--module and thereby a finitely generated $\SI_\QQ$--module. 


Our \autoref{thmA:SIrepstab} implies then that $V$ is uniformly representation stable.
%
\end{proof}

\subsection{Johnson filtration of $\I_{g,1}$}

Next we consider the Johnson filtration $\alpha\I_{g,1}$ of the Torelli subgroups $\I_{g,1}$. 
The proof of \autoref{prop:VICZ alpha} can be used to prove the following analogue.

\begin{prop}\label{prop:SIZ alpha}
$\{\gr_i(\alpha\IA_n)\}_{n\in\NN}$ gives rise to a $\VIC_\ZZ$--module.
\end{prop}

Similar to the Johnson filtration of $\IA_n$, we also get information from the (original) Johnson homomorphism. As explained by Satoh  \cite[Sec 8]{SatSurvey} there is an $\Sp_{2g}\ZZ$--equivariant monomorphism
\[ \gr_i(\alpha\I_{g,1}) \inject \Hom_\QQ(H_1(\Sigma_{g,1};\QQ), \gr_{i+1}(\gamma F_{2g}))\cong \QQ^{2g} \otimes \mathcal L_{i+1}(\QQ^{2g}).\]
By the same arguments used in the proof of \autoref{thm:SIQ beta}, we can deduce the following result.

\begin{thm}\label{thm:SIQ alpha}
$\{\gr_i(\alpha\I_{g,1})\}_{g\in \NN}$ gives rise to a rational $\SI_\QQ$--module.
\end{thm}

\begin{proof}[Proof of \autoref{thmE}]
Let $V$ be the $\SI_\QQ$--module from \autoref{thm:SIQ alpha} such that $V_g = \gr_k(\alpha\I_{g,1})$.

Church and Putman \cite{CP} consider the groups $\I_{g,1}$ as a weak $\FI$-group. They apply their \cite[Thm G]{CP} to prove their \cite[Thm A]{CP}. In the proof of the former theorem in \cite[Claim 2]{CP} it is stated that $W(k)$ is boundedly generated. But $W(k)_g\otimes_\ZZ \QQ$ is $\gr_k(\alpha\I_{g,1}) = V_g$. Because every $V_g$ is finite dimensional (see \cite[Prop 4.4]{CP}), $V$ is a finitely generated $\FI$--module and thus certainly a finitely generated $\SI_\QQ$--module.


Our \autoref{thmA:SIrepstab} implies then that $V$ is uniformly representation stable. 
\end{proof}


\addcontentsline{toc}{section}{References}
\bibliographystyle{halpha}
\bibliography{repstab}

%
%
%
%
%
%
%
    
\end{document}